\documentclass[a4paper,11pt,oneside]{amsart}
\usepackage{amsmath,amsfonts,amssymb,amsthm,amscd,latexsym,euscript}
\usepackage[all]{xy}

\SelectTips{cm}{}

\theoremstyle{plain}
\newtheorem{theorem}{Theorem}[section]
\newtheorem{corollary}[theorem]{Corollary}
\newtheorem{lemma}[theorem]{Lemma}
\newtheorem{proposition}[theorem]{Proposition}

\theoremstyle{definition}
\newtheorem{definition}[theorem]{Definition}
\newtheorem{example}[theorem]{Example}

\newcommand{\Ae}{{\mathcal A}}

\newcommand{\Ce}{{\mathcal C}}
\newcommand{\De}{{\mathcal D}}
\newcommand{\Ee}{{\mathcal E}}
\newcommand{\Fe}{{\mathcal F}}

\newcommand{\Ke}{{\mathcal K}}
\newcommand{\Le}{{\mathcal L}}

\newcommand{\Se}{{\mathcal S}}

\newcommand{\Xe}{{\mathcal X}}
\newcommand{\Ze}{{\mathcal Z}}

\newcommand{\Sets}{{\bf Set}}
\newcommand{\Ssets}{{\bf sSet}}

\newcommand{\Z}{{\mathbb Z}}
\newcommand{\Sph}{{\mathbb S}}
\newcommand{\Map}{{\rm map}}

\newcommand{\Cn}{C(n)}
\newcommand{\Cnpo}{C(n+1)}
\newcommand{\Cnpt}{C(n+2)}
\newcommand{\Czero}{C(0)}
\newcommand{\Cone}{C(1)}
\newcommand{\Exinf}{{\rm Ex}^{\infty}}
\newcommand{\colim}{{\rm colim}}
\newcommand{\Ar}{{\bf Arr}\,}

\numberwithin{equation}{section}

\begin{document}

\title[Definable orthogonality classes are small]{Definable orthogonality classes 
in accessible categories are small}

\author{Joan Bagaria, Carles Casacuberta, \\ A. R. D. Mathias, and Ji\v{r}\'{\i} Rosick\'{y}}
\thanks{The authors were supported by the Spanish 
Ministry of Science and Innovation under grants MTM2007-63277, MTM2008-03389, MTM2010-15831 and MTM2011-25229, 
by the Generalitat de Catalunya under grants 2005~SGR~606, 2005~SGR~738, 2009~SGR~119 and 2009~SGR~187, and by the 
Ministry of Education of the Czech Republic under project MSM0021622409.
This research was supported through the Research in Pairs programme by the
Mathematisches Forschungsinstitut Oberwolfach in 2008.}

\date{\today }

\begin{abstract}
We lower substantially the strength of the assumptions
needed for the validity of certain results in category theory and 
homotopy theory which were known to follow from Vop\v{e}nka's principle.
We prove that the necessary large-cardinal hypotheses depend on the complexity
of the formulas defining the given classes, in the sense of the
L\'evy hierarchy. For example, the statement that, for a class
$\Se$ of morphisms in a locally presentable category~$\Ce$ of structures,
the orthogonal class of objects $\Se^{\perp}$ is a small-orthogonality class
(hence reflective) can be proved in ZFC if
$\Se$ is~{\boldmath$\Sigma_1$}, while it follows from the existence of a proper class
of supercompact cardinals if $\Se$ is~{\boldmath$\Sigma_2$}, and from the
existence of a proper class of what we call $C(n)$\nobreakdash-extendible cardinals if $\Se$ is 
{\boldmath$\Sigma_{n+2}$} for $n\ge 1$. These cardinals form a new hierarchy, and we show that Vop\v{e}nka's
principle is equivalent to the existence of $C(n)$\nobreakdash-extendible cardinals for all~$n$.

As a consequence, we prove that the existence
of cohomological localizations of simplicial sets, a long-standing open problem in algebraic
topology, is implied by the existence of arbitrarily large supercompact cardinals.
This follows from the fact
that the class of $E^*$\nobreakdash-equiv\-al\-ences is {\boldmath$\Sigma_2$}\nobreakdash-def\-in\-able,
where $E$ denotes a spectrum treated as a parameter. 
In contrast with this fact, the class of $E_*$\nobreakdash-equiv\-al\-ences is 
{\boldmath$\Sigma_1$}\nobreakdash-def\-in\-able, from which it follows (as is well known) 
that the existence of homological localizations is provable in~ZFC.
\end{abstract}

\maketitle

\section*{Introduction}
\label{introduction}

The answers to certain questions in category theory turn out to depend on set theory. 
A~typical example is whether every full limit-closed subcategory of a complete category $\Ce$ is reflective. 
On the one hand, there are counterexamples involving the category of topological spaces 
and continuous functions~\cite{TAR}. 
On the other hand, as explained in~\cite{AR}, an affirmative answer to this question 
for locally presentable categories is implied by
a large-cardinal axiom called Vop\v{e}nka's principle
(stating that, for every proper class of structures of the same type, 
there exists a nontrivial elementary embedding between two of them).

Large cardinals were used in a similar way in~\cite{CSS}
to show that the existence of cohomological localizations, a famous unsolved problem,
follows from Vop\v{e}nka's principle.
Other relevant consequences of Vop\v{e}nka's principle in algebraic topology 
were found in \cite{CCh}, \cite{CGR}, \cite{Ch2}, \cite{RT}.
However, the precise consistency strength of many implications of
this axiom in category theory or homotopy theory is not known,
and in some cases the question of whether such statements are provable in ZFC
remains unanswered. A~relevant step in this direction was made in~\cite{Prz}.

In another direction,
it was pointed out in \cite{BCM} that certain results about accessible categories
that follow from Vop\v{e}nka's principle are still true under much weaker large-cardinal assumptions.
This claim is based on the following finding, which is the subject of the present article:
\emph{the assumptions needed to infer reflectivity
or smallness of orthogonality classes in accessible categories
may depend on the complexity of the formulas in the language of set theory defining these classes}.
Here ``complexity'' is meant in the sense of the L\'evy hierarchy \cite[Ch.\,13]{J2}.
Recall that $\Sigma_n$ formulas and $\Pi_n$ formulas are defined inductively as follows: 
$\Pi_0$ formulas are the same as $\Sigma_0$ formulas, namely formulas in which
all quantifiers are bounded; $\Sigma_{n+1}$ formulas are of the form 
$\exists x\, \varphi$ where $\varphi$ is~$\Pi_n$, and $\Pi_{n+1}$ formulas are of the form 
$\forall x\, \varphi$ where $\varphi$ is~$\Sigma_n$. 

For example, as we prove in this article, 
if $\Se$ is a full limit-closed subcategory of
a locally presentable category~$\Ce$ of structures,
and $\Se$ can be defined with a $\Sigma_2$ formula (possibly with parameters),
then the existence of a proper class of supercompact cardinals suffices
to ensure reflectivity of~$\Se$. Moreover, remarkably, if $\Se$ can be defined with a $\Sigma_1$ formula, then the reflectivity of $\Se$ is provable in~ZFC.

In case of a more complex definition of~$\Se$, 
its~reflectivity follows from the existence of a proper class of what 
we call \emph{$\Cn$\nobreakdash-extendible cardinals}, for some~$n$. These cardinals form a natural hierarchy
ranging from extendible cardinals \cite[20.22]{J2} when $n=1$ to Vop\v{e}nka's principle.
Indeed, as stated in Corollary~\ref{corollary2bis} below, 
Vop\v{e}nka's principle is equivalent to the claim that
there exists a $\Cn$\nobreakdash-extendible cardinal for every~$n<\omega$.
We denote by $\Cn$ the proper class of cardinals $\alpha$ such that $V_{\alpha}$
is a $\Sigma_n$\nobreakdash-elementary submodel of the set-theoretic universe~$V$, and say that a cardinal $\kappa$
is $\Cn$\nobreakdash-extendible if $\kappa\in\Cn$ and for all $\lambda > \kappa$ in $\Cn$
there is an elementary embedding $j\colon V_{\lambda}\to V_{\mu}$ for some $\mu\in\Cn$
with critical point~$\kappa$, such that $j(\kappa)\in\Cn$ and $j(\kappa)>\lambda$.

By way of this approach, we prove that the existence of cohomological
localizations of simplicial sets follows from the existence of a proper class of supercompact cardinals.
This result uses the fact, proved in Theorem~\ref{cohomology} below,
that for every (Bousfield--Friedlander) spectrum $E$ the class of $E^*$\nobreakdash-acyclic 
simplicial sets (where $E^*$ denotes the reduced cohomology theory represented by~$E$) can be defined by means 
of a $\Sigma_2$ formula with $E$ as a parameter.
However, the class of $E_*$\nobreakdash-acyclic simplicial sets (where $E_*$ now denotes
homology) can be defined with a $\Sigma_1$ formula. This is consistent with the fact that the
existence of homological localizations can be proved in~ZFC, as done indeed by Bousfield
in~\cite{BouTop}; see also~\cite{AF}.

The reason why classes of homology acyclics have lower complexity than classes of cohomology acyclics is that, for a fibrant simplicial set $Y$ with basepoint, the statement ``all pointed maps $f\colon \Sph^n\to Y$ are nullhomotopic'', where $\Sph^n$ is the simplicial $n$\nobreakdash-sphere, is absolute between transitive models of~ZFC,
since a simplicial map $\Sph^n\to Y$ is determined by a single $n$\nobreakdash-simplex of $Y$ satisfying certain conditions expressible in terms of $Y$ with bounded quantifiers; cf.\,\cite[3.6]{May}.
However, if $X$ and $Y$ are simplicial sets with basepoints $x_0$ and $y_0$,
then the statement ``all pointed maps $f\colon X\to Y$ are nullhomotopic''
involves unbounded quantifiers, since it is formalized, for example, by stating that
\[
\forall f \, (\mbox{$f$ is a map from $X$ to $Y$} \to \exists h\, (\mbox{$h$ is a homotopy from $f$ to $y_0$})).
\]

Therefore, for a spectrum~$E$, there might exist $E^*$\nobreakdash-acyclic simplicial sets in a 
transitive model of ZFC containing $E$ that fail to be $E^*$\nobreakdash-acyclic in some larger model,
while the class of $E_*$\nobreakdash-acyclic simplicial sets is absolute.
See Section~\ref{cohomologicallocalizations} for a detailed discussion of these facts.

Another consequence of this article is that the main theorem of \cite{BCM}
can now be proved for reflections, not necessarily epireflections.
Thus, if there are arbitrarily large supercompact cardinals, then every reflection $L$
on an accessible category of structures is an $\Fe$\nobreakdash-reflection for some set of morphisms~$\Fe$,
provided that the class of $L$\nobreakdash-equivalences is {\boldmath$\Sigma_{2}$};
see Corollary~\ref{onemore} below.
(Boldface types {\boldmath$\Sigma_n$} or {\boldmath$\Pi_n$}
are used to denote the fact that the corresponding formulas may contain parameters.)

We also prove that the Freyd--Kelly orthogonal subcategory problem \cite{FK},
asking if $\Se^{\perp}$ is reflective for a class of morphisms~$\Se$ in a suitable category,
has an affirmative answer in ZFC for {\boldmath$\Sigma_1$} classes in locally presentable categories of structures.
It is also true for {\boldmath$\Sigma_2$} classes if a proper class of super\-compact
cardinals is assumed to exist, and for {\boldmath$\Sigma_{n+2}$} classes if there is a proper class
of $\Cn$\nobreakdash-exten\-dible cardinals for~$n\ge 1$. We say that $\Se$ is \emph{definable with 
sufficiently low complexity} to encompass all these cases in a single phrase.

Essentially the same arguments hold in the homotopy category of simplicial sets,
hence yielding a simpler and more accurate answer than in \cite{CSS} (where Vop\v{e}nka's principle
was used) to Farjoun's question in \cite{Far1} of whether every homotopy reflection
on simplicial sets is an $f$\nobreakdash-localization for some map~$f$.
Localizations with respect
to sets of maps were constructed in \cite{BouJPAA}, \cite{Far2}, \cite{Hir}, and the extension
to proper classes of maps was carried out in \cite{CSS} using Vop\v{e}nka's principle.
Here we prove that localizations with respect to proper classes of maps exist whenever
the given classes are definable with sufficiently low complexity.

We warn the reader that
in this article, as well as in~\cite{BCM}, complexity of classes
of objects or morphisms in an accessible category $\Ce$
is meant under the assumption that $\Ce$ is accessibly embedded into a category of structures.
This happens canonically with the category of simplicial sets and with the category
of Bousfield--Friedlander spectra, or, more generally, with categories of models of
basic theories in any language.
Terminology and background can be found in \cite[5.B]{AR}, where it is proved that every
accessible category is equivalent to one which is accessibly embedded into a
category of structures.

\bigskip

\noindent
\textbf{Acknowledgements} \
We are much indebted to the referee for a deep and careful reading of the manuscript and a number of pertinent corrections.

\section{Categories of structures}
\label{categoriesofstructures}

Most of the results in this article refer to categories of \emph{structures} (possibly many-sorted,
in a language of any cardinality). For the convenience of the reader, we start by recalling 
terminology and background about structures and models in this section.
Additional details can be found, among many other sources, in~\cite[Ch.\,5]{AR} and~\cite[Ch.\,12]{J2}.

For a regular cardinal~$\lambda$,
a \emph{$\lambda$\nobreakdash-ary $S$\nobreakdash-sorted signature} $\Sigma$
consists of a set $S$ of \emph{sorts}, a set $\Sigma_{\rm op}$ of \textit{operation symbols},
another set $\Sigma_{\rm rel}$ of \textit{relation symbols}, and an \emph{arity} function that assigns to each
operation symbol an ordinal $\alpha<\lambda$, a sequence $\langle s_i : i\in\alpha \rangle$ of \emph{input sorts} 
and an \emph{output sort} $s\in S$,
and to each relation symbol an ordinal $\beta<\lambda$ and
a sequence of sorts $\langle s_j : j\in\beta \rangle$.
An operation symbol with $\alpha=\emptyset$ is called a \emph{constant symbol}.
A~signature $\Sigma$ is called \emph{operational} if $\Sigma_{\rm rel}=\emptyset$ and \emph{relational} if $\Sigma_{\rm op}=\emptyset$.

Given an $S$\nobreakdash-sorted signature~$\Sigma$, a \emph{$\Sigma$\nobreakdash-structure} is a triple
\[
X=\left\langle \{ X_s : s\in S\},\, \{\sigma_X : \sigma\in\Sigma_{\rm op}\},\, \{\rho_X : \rho\in\Sigma_{\rm rel}\} \right\rangle
\]
consisting of an \emph{underlying $S$\nobreakdash-sorted set} or \emph{universe},  denoted by $\{X_s : s\in S\}$ or $(X_s)_{s\in S}$, together with a function 
\[
\sigma_X\colon \prod_{i\in\alpha}X_{s_i}\longrightarrow X_s
\]
for each operation symbol $\sigma\in\Sigma_{\rm op}$ 
of arity $\langle s_i : i\in\alpha \rangle\to s$ (including
a distinguished element of $X_s$ for each constant symbol of sort~$s$),
and a set 
\[
\rho_X\subseteq \prod_{j\in\beta}X_{s_j}
\]
for each relation symbol $\rho\in\Sigma_{\rm rel}$ of arity $\langle s_j : j\in\beta\rangle$. 

A \emph{homomorphism} $f\colon X\to Y$ between two
$\Sigma$\nobreakdash-structures is an $S$\nobreakdash-sorted function 
$( f_s \colon X_s\to Y_s)_{s\in S}$ preserving operations and relations. 
For each signature~$\Sigma$, the category of $\Sigma$\nobreakdash-structures and
their homomorphisms will be denoted by~${\bf Str}\,\Sigma$.

Given a $\lambda$\nobreakdash-ary $S$\nobreakdash-sorted signature~$\Sigma$, 
the \emph{language} $\Le_{\lambda}(\Sigma)$ consists of sets of \emph{variables}, \emph{terms}, and \emph{formulas},
which are defined as follows. There is a family $W=\{ W_s : s\in S \}$ 
of sets of cardinality~$\lambda$, the elements of $W_s$ being \emph{variables} of sort~$s$. 
One defines \emph{terms} by declaring that each variable is a term and, for each operation symbol $\sigma\in\Sigma_{\rm op}$ of arity 
$\langle s_i : i\in\alpha \rangle\to s$ and each collection of terms $\tau_i$
of sort~$s_i$, the expression $\sigma(\tau_i)_{i\in\alpha}$
is a term of sort~$s$. \emph{Atomic formulas} are expressions of the form $\tau_1=\tau_2$ and
$\rho(\tau_j)_{j\in\beta}$, where $\rho\in\Sigma_{\rm rel}$ is a relation symbol of arity 
$\langle s_j : j\in\beta\rangle$
and each $\tau_j$ is a term of sort $s_j$ with $j\in\beta$. \emph{Formulas} are built in finitely many steps from the atomic formulas 
by means of logical connectives and quantifiers. Thus, if $\{ \varphi_i :i\in I \}$ are formulas and $|I|<\lambda$,
then so are the conjunction $\bigwedge_{i\in I}\varphi_i$ and the disjunction $\bigvee_{i\in I}\varphi_i$.
Quantification is allowed over sets of variables of cardinality smaller than~$\lambda$; that is,
$\left(\forall (x_i)_{i\in I}\right) \varphi$ and $\left(\exists (x_i)_{i\in I}\right) \varphi$ are formulas if 
$\varphi$ is a formula and $|I|<\lambda$. 

Variables that appear unquantified in a formula are called \emph{free}. If a formula is denoted by
$\varphi(x_i)_{i\in I}$, it is meant that each $x_i$ is a free variable.

Each language $\Le_{\lambda}(\Sigma)$ determines a \emph{satisfaction relation}
between $\Sigma$\nobreakdash-struc\-tures and formulas with an assignment for their free variables.
If $\varphi(x_i)_{i\in I}$ is a formula where each $x_i$ is a free variable of sort $s_i$ 
and $X$ is a $\Sigma$\nobreakdash-structure, a \emph{variable assignment}, denoted by $x_i\mapsto a_i$, 
is a function $a\colon I\to\cup_{s\in S}\,X_s$ such that $a(i)\in X_{s_i}$ for all~$i$.
Satisfaction of a formula $\varphi$ in a $\Sigma$\nobreakdash-structure $X$
is defined inductively, starting with the atomic formulas and quantifying over subsets of $\cup_{s\in S}\,X_s$
of cardinality smaller than~$\lambda$; see~\cite[\S 5.26]{AR} for details.
We write $X\models\varphi(a_i)_{i\in I}$ if $\varphi$ is satisfied in $X$ under 
an assignment $x_i\mapsto a_i$ for all its free variables~$x_i$.

A formula without free variables is called
a \emph{sentence}. A set of sentences is called a \emph{theory}.
A \emph{model} of a theory $T$ in a language $\Le_{\lambda}(\Sigma)$ is a $\Sigma$\nobreakdash-structure satisfying all
sentences of~$T$. For each theory~$T$, we denote by ${\bf Mod}\,T$ the full subcategory of
${\bf Str}\,\Sigma$ consisting of all models of~$T$.

A language $\Le_{\lambda}(\Sigma)$ is called \emph{finitary} if $\lambda=\omega$ (the least
infinite cardinal); otherwise it is \emph{infinitary}.
An especially important finitary language is the \emph{language of set theory}. This is the 
first-order finitary language corresponding to the signature with one sort, namely ``sets'', 
and one binary relation symbol (``membership''). Hence the atomic formulas 
are $x=y$ and $x\in y$, where $x$ and $y$ are sets.

Define, recursively on the class of ordinals, $V_0=\emptyset$,
$V_{\alpha +1}={\mathcal P}(V_{\alpha})$ for all~$\alpha$, 
where ${\mathcal{P}}$ denotes the power-set operation, and 
$V_{\lambda}=\bigcup_{\alpha < \lambda}V_{\alpha}$ 
if $\lambda$ is a limit ordinal. 
Then every set is an element of some~$V_{\alpha}$; see~\cite[Lemma~9.3]{J1} or~\cite[Lemma~6.3]{J2}.
The \emph{rank} of a set~$X$ is the least ordinal $\alpha$ such that $X\in V_{\alpha +1}$. 
Hence $V_{\alpha}$ is the set of all sets whose rank is less than~$\alpha$.
The \emph{universe} $V$ of all sets is the union of $V_{\alpha}$ for all ordinals~$\alpha$.

Everything in this article is formulated in ZFC (Zermelo--Fraenkel set theory with the axiom of choice). 
Thus, a \emph{class} consists of all sets for which a certain formula of the language of set theory is satisfied,
possibly with parameters. More precisely, a class $\Ce$ is \emph{defined} by a formula
$\varphi(x,y_1,\dots,y_n)$ with \emph{parameters} $p_1,\dots,p_n$ if 
\[
\Ce=\{x : \varphi(x,p_1,\dots,p_n)\}, 
\]
where satisfaction, if unspecified, is meant in the universe~$V$.
The sets $p_1,\dots,p_n$ are fixed values of $y_1,\dots,y_n$ under every variable assignment.
To simplify the notation, we often replace $p_1,\dots,p_n$ by a single parameter $p=\{p_1,\dots,p_n\}$.
A class which is not a set is called a \emph{proper~class}.
Each set $A$ is definable with $A$ itself as a parameter by \hbox{$A=\{x : x\in A\}$}.

In this article, a \emph{model of ZFC} will be a pair $\langle M,\in\rangle$ where $M$ is a
set or a proper class and $\in$ is the restriction of the membership relation to~$M$,
in which the formalized ZFC axioms are satisfied. Thus, if we neglect the fact that $M$ can be a proper class, 
we may view $\langle M,\in\rangle$ as a $\Sigma$\nobreakdash-structure where $\Sigma$ is the relational signature
of the language of set theory, and in fact a model of the theory consisting of the formalized ZFC axioms.
In particular, $\langle V,\in\rangle$ itself is such a model.

A class $M$ is \emph{transitive} if every element of an element of $M$ is an element of~$M$.
We shall always assume that models of ZFC are transitive,
but not necessarily inner (a model is called \emph{inner} 
if it is transitive and contains all the ordinals).

\section{The L\'evy hierarchy}
\label{Levyhierarchy}

In this section we specialize to the language of set theory.
Thus, given two classes $M\subseteq N$, 
we say that a formula $\varphi(x_1,\ldots,x_k)$
is \emph{absolute between $M$ and~$N$} if, for all $a_1,\ldots,a_k$ in~$M$,
\[
N\models \varphi(a_1,\ldots,a_k)\; \mbox{ if and only if }\; M\models \varphi(a_1,\ldots,a_k).
\]

We say that a formula $\varphi(x_1,\ldots,x_k)$ is \emph{upward absolute} for transitive
models of some theory~$T$ if,
given any two such models $M\subseteq N$ and given $a_1,\ldots,a_k\in M$ for which
$\varphi(a_1,\ldots,a_k)$ is true in~$M$, 
$\varphi(a_1,\ldots,a_k)$ is also true in~$N$.
And we say that $\varphi$ is \emph{downward absolute}
if, in the same situation, if $\varphi(a_1,\dots,a_k)$ holds in $N$ then it holds in~$M$.
A~formula is \emph{absolute} if it is both upward and downward absolute.
If $T$ is unspecified, then it should be understood that $T$ is by default the set of
all formalized ZFC axioms. If it is meant, on the contrary, that $T=\emptyset$, then we speak of absoluteness between transitive classes.

A class $\Ce$ is \emph{upward absolute} between transitive classes $M\subseteq N$ if it is definable, possibly with a set $p$ of parameters, by a formula that is upward absolute between $M$ and~$N$. \emph{Downward absolute} classes are defined analogously, and we say that $\Ce$ is \emph{absolute} between $M$ and $N$ if it is upward absolute and downward absolute, hence allowing the possibility that 
\[
\Ce=\{x:\varphi(x,p)\}=\{x:\psi(x,p)\}
\]
where $\varphi$ is upward absolute and $\psi$ is downward absolute. In this situation, $N\models x\in\Ce$ if and only if $M\models x\in\Ce$, assuming that $p\in M$.

The following terminology is due to L\'evy; see \cite[Ch.\,13]{J2}.
A~formula of the language of set theory is said to be 
$\Sigma_0$ if all its quantifiers are bounded, that is, of the form $\exists x\in a$ 
or $\forall x\in a$.
Then \emph{$\Sigma_n$ formulas} and \emph{$\Pi_n$ formulas} are defined inductively as follows: 
$\Pi_0$ formulas are the same as $\Sigma_0$ formulas; $\Sigma_{n+1}$ formulas are of the form 
$(\exists x_1\dots x_k)\, \varphi$, where $\varphi$ is $\Pi_n$; 
and $\Pi_{n+1}$ formulas are of the form $(\forall x_1\dots x_k)\, \varphi$, where $\varphi$ is~$\Sigma_n$. 
We say that a formula is $\Sigma_n\wedge\Pi_n$ if it is a conjunction of a $\Sigma_n$ formula and a $\Pi_n$ formula.

Classes can be defined by distinct formulas and, more generally, properties and
mathematical statements can be formalized in the language of set theory in many different ways.
We say that a class $\Ce$ is \emph{{\boldmath$\Sigma_n$}\nobreakdash-definable} (or, shortly, that $\Ce$ is~{\boldmath$\Sigma_n$}) 
if there is a $\Sigma_n$ formula $\varphi(x,y)$ such that
$\Ce=\{ x : \varphi (x,p) \}$ for a set $p$ of parameters.
Similarly, a class is {\boldmath$\Pi_n$}
if it can be defined by some $\Pi_n$ formula with parameters.
A~class is called {\boldmath$\Delta_n$} if it is both {\boldmath$\Sigma_n$} and~{\boldmath$\Pi_n$}.
For notational convenience, if no parameters are involved, then we write that a class $\Ce$ is $\Sigma_n$, $\Pi_n$ or $\Delta_n$, using lightface types. 

The same terminology is used with statements or informal expressions; for example, ``$\lambda$ is a cardinal'' is a $\Pi_1$ statement \cite[Lemma~13.13]{J2}, while ``$f$ is a function'', ``$\alpha$~is an ordinal'' or ``$\omega$~is the least nonzero limit ordinal'' are $\Delta_0$ statements \cite[Lemma~12.10]{J2}.

If a class $\Ce$ is {\boldmath$\Sigma_1$} with a set $p$ of parameters, 
then it is upward absolute for transitive classes containing~$p$.
In fact, given a $\Sigma_1$ formula $\exists x\,\varphi(x,y)$ where $\varphi$ is $\Sigma_0$
and given a set $p$ of parameters, suppose that
$M\subseteq N$ are transitive classes with $p\in M$. Then, if 
$M\models \exists x\,\varphi(x,p)$, we may infer that $N\models \exists x\,\varphi(x,p)$
as well, since if $a\in M$ witnesses that $\varphi(a,p)$ holds in~$M$, then $a\in N$ and 
$\varphi(a,p)$ also holds in~$N$, since $\varphi$ is absolute.

Conversely, if a class $\Ce$ is upward absolute for transitive models of some finite fragment
${\rm ZFC}^*$ of~ZFC, then it is~{\boldmath$\Sigma_1$}. 
To prove this claim, suppose that $\Ce$ is defined by a formula $\varphi(x,y)$ that is upward absolute for transitive models
of ${\rm ZFC}^*$ with a set $p$ of parameters. Then $\Ce$ is also defined by the following $\Sigma_1$ formula:
\begin{equation}
\label{trick}
\exists M\, [\mbox{$M$ is transitive} \, \wedge \, \{x,p\}\subset M \, \wedge \, 
M\models \left(\varphi(x,p)\wedge\left(\textstyle\bigwedge {\rm ZFC}^*\right)\right)].
\end{equation}
Indeed, if $a\in\Ce$ then $\varphi(a,p)$ holds in~$V$, and it follows from the 
Reflection Principle \cite[Theorem~12.14]{J2} that there is an ordinal $\alpha$ with $\{a,p\}\in V_\alpha$ 
such that $V_\alpha \models \varphi(a,p)$ and all the sentences in the finite set ${\rm ZFC}^*$ 
are satisfied in~$V_{\alpha}$, so $V_\alpha$ witnesses~\eqref{trick}. 
And, if a set $M$ witnesses \eqref{trick} for some variable assignment $x\mapsto a$, 
then, since $\varphi(x,y)$ is upward absolute for transitive models of~${\rm ZFC}^*$, 
we infer that $\varphi(a,p)$ holds in~$V$, that is, $a\in \Ce$.

Similarly, if a class $\Ce$ is defined by a $\Pi_1$ formula with parameters, then it is
downward absolute for transitive classes containing the parameters, 
and, if $\Ce$ is downward absolute for transitive models of some finite
fragment of~ZFC, then it is~{\boldmath$\Pi_1$}, analogously as in~\eqref{trick}.
We conclude that {\boldmath$\Delta_1$}~classes are absolute for transitive classes containing the parameters.

The following are examples of nonabsoluteness which will be relevant in this article.

\begin{example}
The class of topological spaces is~$\Pi_1$, since
the union of every collection of open sets must be open.
Thus, a topology on a set $X$ in some model of ZFC
may fail to be a topology on $X$ in a larger model.
However, the class of simplicial sets is~$\Delta_0$ 
(see Section~\ref{cohomologicallocalizations}).
\end{example}

\begin{example}
\label{example2}
Let $\mathcal{C}$ be the class of all abelian groups of the form 
$\Z^{\kappa}$, where $\kappa$ is a cardinal. 
Then $A\in \mathcal{C}$ if and only if
\[
\exists x\,(\mbox{$x$ is a cardinal}\wedge \forall y\,(y\in A\leftrightarrow \mbox{$y$ is a function
from $x$ to $\Z$})),
\]
which is a $\Sigma_2$ formula, since the expression written within the outer parentheses is~$\Pi_1$.
In every model of ZFC with measurable cardinals, the following sentence is true:
\[
\begin{array}{c}
\exists \kappa\, \exists f\, (\mbox{$\kappa$ is an infinite cardinal} \, \wedge \, 
\mbox{$f$ is a group homomorphism} \\[0.1cm] \mbox{from $\Z^{\kappa}$ to $\Z$} \, \wedge \,
f(\Z^{<\kappa})=0 \, \wedge \, f\ne 0),
\end{array}
\]
while if this holds then the smallest $\kappa$ with this property is measurable, according to~\cite{E}; 
see \cite{EM} for further details.
Therefore, this sentence is false in a model of ZFC without measurable cardinals while it is true
in a model of ZFC with measurable cardinals.
\end{example}

\begin{example}
\label{example3}
For a cardinal $\lambda$ and a set~$X$, we denote by ${\mathcal P}_{\lambda}(X)$
the set of all subsets of $X$ whose cardinality is smaller than~$\lambda$.
Note first that, although the statement ``$A$ is a subset of~$B$'' is~$\Delta_0$,
the statement ``$A$ is the set of all subsets of~$B$''
is formalized with the following $\Pi_1$ formula: 
\[
\forall a\in A\,(a\subseteq B) \, \wedge \, \forall x\,(x\subseteq B \to x\in A).
\]
This statement cannot be formalized with any upward absolute formula, since,
if we pick a countable transitive model $M$ of~ZFC and $A$ is the set of all subsets of the natural numbers $\mathbb{N}$ in~$M$,
then $A$ cannot be the set of all subsets of $\mathbb{N}$ in the universe~$V$, since $A$ is countable.

The assertion ``$x$ is finite'' is $\Delta_1$, since
it is equivalent to the statement that there exists a bijection between $x$
and a finite ordinal (which is~$\Sigma_1$) and it is also equivalent
to the statement that every injective function from $x$ to itself is surjective
(which is~$\Pi_1$). Note also that, if a set $x$ is finite and each of its
elements belongs to a model $M$ of~ZFC, then we may infer that $x\in M$ 
using the pairing and union axioms.
 From this fact it follows that the statement $A={\mathcal P}_{\omega}(B)$
---that is, ``$A$ is the set of all finite subsets of~$B$''---
is absolute for transitive models of a suitable finite fragment of~ZFC, hence~$\Delta_1$.
Nevertheless,
if $M$ and $N$ are just transitive classes with 
$M\subset N$ and $B\in M$, it can happen that the claim
``${\mathcal P}_{\omega}(B)$ exists'' is true in $N$ but not in~$M$, as discussed in~\cite[Sections 5 and~6]{ARDM}.

For a cardinal $\lambda>\omega$, the expression $A={\mathcal P}_{\lambda}(B)$ 
can be formalized by claiming that $\lambda$ is a cardinal and
$\forall x\,(x\in A \leftrightarrow (x\subseteq B\,\wedge\,|x|<\lambda))$.
The clause $|x|<\lambda$ is, on one hand, equivalent to
\[
(\exists\alpha\in\lambda)\,\exists f\,(\mbox{$f$ is a bijective function from $x$ to $\alpha$}),
\]
which is~$\Sigma_1$, and on the other hand it is the negation of $\lambda\le |x|$, hence
equivalent to the $\Pi_1$ claim that there is no injective function from $\lambda$ to~$x$. 
Therefore, the statement $A={\mathcal P}_{\lambda}(B)$ is~$\Pi_1$.
\end{example}

\section{Complexity of categories}
\label{definablecategories}

In order to simplify expressions, if $\Ce$ is a category
we shall denote by $X\in\Ce$ the statement that $X$ is an object of $\Ce$
and by $f\in\Ce(X,Y)$ the claim that $X$ and $Y$ are objects of $\Ce$
and $f$ is a morphism from $X$ to~$Y$.

\begin{definition}
\label{definablecategory}
{\rm
For $n\ge 0$,
a category $\Ce$ is called {\boldmath$\Sigma_n$}\nobreakdash-\emph{definable} (shortly,~{\boldmath$\Sigma_n$}) with a set $p$ of parameters
if there is a $\Sigma_n$ formula $\varphi$ of the language
of set theory such that $\varphi(X,Y,Z,f,g,h,i,p)$
is true if and only if $f\in\Ce(X,Y)$, $g\in\Ce(Y,Z)$, $h$ is the composite of $f$ and~$g$, and $i$ is the identity of~$X$.
}
\end{definition}

If a category $\Ce$ is {\boldmath$\Sigma_n$} with a set $p$ of parameters,
then there are $\Sigma_n$ formulas
$\psi_{\rm Ob}(x,y)$ and $\psi_{\rm Mor}(x,y,z,t)$
such that
$\psi_{\rm Ob}(X,p)$ is true if and only if $X\in\Ce$
and
$\psi_{\rm Mor}(X,Y,f,p)$ is true if and only if $f\in\Ce(X,Y)$.
Specifically, from a formula $\varphi$ as in~Definition~\ref{definablecategory}
we can choose $\psi_{\rm Mor}(x,y,z,t)$ to be
$
\exists i\,\varphi(x,x,y,i,z,z,i,t)
$,
and next choose $\psi_{\rm Ob}(x,y)$ to be $\exists z\,\psi_{\rm Mor}(x,x,z,y)$.

If $\Ce$ is~{\boldmath$\Sigma_n$}, then the statement $F=\Ce(X,Y)$ 
is formalized with the following $\Sigma_n\wedge\Pi_{n}$ formula:
\[
(\forall f\in F)\, f\in\Ce(X,Y)\, \wedge\, \forall g\, (g\in\Ce(X,Y)\to g\in F).
\]

We say that a category is \emph{{\boldmath$\Pi_n$}}
for $n\ge 0$ if there are $\Pi_n$ formulas
defining its objects, morphisms, composition and identities. 
A category will be called {\boldmath$\Delta_n$}
if it is both {\boldmath$\Sigma_n$} and~{\boldmath$\Pi_n$}.

A~category is \emph{upward absolute} for transitive classes if its objects, morphisms, composition and identities can be defined by formulas that are upward absolute for transitive classes. \emph{Downward absolute} categories are defined in the same way, and a category will be called \emph{absolute} if it is both upward absolute and downward absolute. Thus, {\boldmath$\Delta_1$} categories are absolute for transitive classes containing the involved parameters.

If $\Ce$ is a subcategory of the category of sets,
then composition and identities in $\Ce$ are prescribed by those of sets. 
Therefore, the complexity of a subcategory of sets is the same if defined
as in Definition~\ref{definablecategory} or if simply treated as a class
of sets together with a class of functions.

Many important categories which cannot be embedded into $\Sets$ have nevertheless
a complexity in our sense. For example, the homotopy category of simplicial sets 
cannot be embedded into $\Sets$ according to~\cite{F}, and yet it can be defined
with a~$\Sigma_2$ formula, since $\mu$ is a morphism from $X$ to $Y$ if and only if there exists 
a simplicial map $f$ from $X$ to a fibrant replacement of $Y$ 
such that $\mu$ is the set of all simplicial maps homotopic to~$f$,
and composition is defined accordingly
(fibrant replacements are discussed in Section~\ref{cohomologicallocalizations}).

For a category $\Ce$ and an object $A$ of~$\Ce$, we denote by $(\Ce\downarrow A)$
the \emph{slice category} whose objects are pairs $\langle X,f\rangle$ where
$f\in\Ce(X,A)$ and whose morphisms $\langle X,f\rangle\to\langle X',f'\rangle$
are morphisms $g\in\Ce(X,X')$ such that $f=f'\circ g$.
Dually, the objects of the \emph{coslice category} 
$(A\downarrow\Ce)$ are pairs $\langle X,f\rangle$ where $f\in\Ce(A,X)$, with 
corresponding morphisms.
Both $(\Ce\downarrow A)$ and $(A\downarrow\Ce)$ are definable with the same complexity as~$\Ce$,
with $A$ as an additional parameter.
Slice and coslice categories are (non-full) subcategories of the
\emph{category of arrows} $\Ar\Ce$, whose objects are triples $\langle A,B,f\rangle$ 
with $f\in\Ce(A,B)$ and where a morphism $f\to g$ is a commutative square
\[
\xymatrix{ 
A\ar[r]\ar[d]_f & C\ar[d]^g \\
B\ar[r] & D\rlap{.}
}
\]

\begin{lemma}
\label{relativizing}
If $\Sigma$ is any signature, then there is a signature $\Sigma'$ such that
$\Ar{\bf Str}\,\Sigma$ fully embeds into ${\bf Str}\,\Sigma'$, and,
if $A$ is a $\Sigma$\nobreakdash-structure, then there is a signature
$\Sigma''$ such that $(A\downarrow{\bf Str}\,\Sigma)$ fully embeds into ${\bf Str}\,\Sigma''$.
In both cases, the embedding preserves complexity.
\end{lemma}

\begin{proof}
Let $S$ be the set of sorts of~$\Sigma$. 
Consider a new set of sorts $S'$ with two elements $s^0$ and $s^1$ for each $s\in S$, and let $\Sigma'$ be the $S'$\nobreakdash-sorted signature with the following operation symbols and relation symbols. The set $\Sigma'_{\rm op}$ has two symbols $\sigma^0$ and $\sigma^1$ of respective arities $\langle (s_i)^0:i\in\alpha\rangle \to s^0$ and $\langle (s_i)^1:i\in\alpha\rangle \to s^1$ 
for each symbol $\sigma\in\Sigma_{\rm op}$ of arity $\langle s_i:i\in\alpha\rangle \to s$, and an additional symbol $\mu_s$ of arity $s^0\to s^1$ for each $s\in S$.
The set $\Sigma'_{\rm rel}$ has two symbols $\rho^0$ and $\rho^1$ of respective arities $\langle (s_j)^0:j\in\beta\rangle$ and $\langle (s_j)^1:j\in\beta\rangle$ for each symbol $\rho\in\Sigma_{\rm rel}$ of arity $\langle s_j:j\in\beta\rangle$.

Then a $\Sigma'$\nobreakdash-structure is a pair of
$\Sigma$\nobreakdash-structures $X^0$ and $X^1$ together with an $S$\nobreakdash-sorted function 
$\mu\colon X^0\to X^1$. Therefore, $\Ar{\bf Str}\,\Sigma$ is canonically
isomorphic to the full subcategory of ${\bf Str}\,\Sigma'$
whose objects are triples $\langle X^0,X^1,\mu\rangle$ for which $\mu$ is a homomorphism of $\Sigma$\nobreakdash-structures.

For the second claim, define,
as in \cite[1.57(2)]{AR}, a signature $\Sigma''$ by adding to $\Sigma$ a new relation symbol
$\rho_a$ of arity $s$ for each element $a\in A_s$.
It then follows that $(A\downarrow{\bf Str}\,\Sigma)$ is canonically isomorphic
to the full subcategory of ${\bf Str}\,\Sigma''$ whose objects are
those $Y\in{\bf Str}\,\Sigma$ for which $(\rho_a)_Y$ consists
of a single element of $Y_s$ for each $a\in A_s$ and the function 
$\rho_Y\colon A\to Y$ given by $\rho_Y(a)=(\rho_a)_Y$ is a homomorphism of $\Sigma$\nobreakdash-structures. 

Both embeddings preserve complexity due to their canonical nature.
In more detail, suppose given a {\boldmath$\Sigma_n$} class $\Fe$ of
objects in $\Ar{\bf Str}\,\Sigma$. Then its image $\Fe'$ in ${\bf Str}\,\Sigma'$ is defined as the class of $\Sigma'$\nobreakdash-structures
\begin{multline*}
X=
\langle
\{ X_{s^0}:s\in S \}\cup
\{ X_{s^1}:s\in S \}, \\
\{ (\sigma^0)_X:\sigma\in\Sigma_{\rm op} \}\cup
\{ (\sigma^1)_X:\sigma\in\Sigma_{\rm op} \}\cup
\{ (\mu_s)_X:s\in S \}, \\
\{ (\rho^0)_X:\rho\in\Sigma_{\rm rel} \}\cup
\{ (\rho^1)_X:\rho\in\Sigma_{\rm rel} \}
\rangle
\end{multline*}
for which the triple consisting of
\begin{align*}
X^0 & =
\langle
\{ X_{s^0}:s\in S \}, \,
\{ (\sigma^0)_X:\sigma\in\Sigma_{\rm op} \}, \,
\{ (\rho^0)_X:\rho\in\Sigma_{\rm rel} \}
\rangle,
\\[0.1cm]
X^1 & =
\langle
\{ X_{s^1}:s\in S \}, \,
\{ (\sigma^1)_X:\sigma\in\Sigma_{\rm op} \}, \,
\{ (\rho^1)_X:\rho\in\Sigma_{\rm rel} \}
\rangle,
\end{align*}
together with the $S$\nobreakdash-sorted function $f\colon X^0\to X^1$ given by $f_s=(\mu_s)_X$ for all $s\in S$ is in the class~$\Fe$. Hence, $\Fe'$ is also {\boldmath$\Sigma_n$}, and analogously with~{\boldmath$\Pi_n$}. 

The argument for $(A\downarrow{\bf Str}\,\Sigma)$ is similar.
\end{proof}

\begin{proposition}
\label{nonabsolute}
If $\Sigma$ is a $\lambda$\nobreakdash-ary signature for a regular cardinal~$\lambda$, then the following assertions hold:
\begin{itemize}
\item[{\rm (a)}]
The category ${\bf Str}\,\Sigma$ of $\Sigma$\nobreakdash-structures
is {\boldmath$\Pi_1$} with parameters $\{\lambda,\Sigma\}$, and it is absolute
between transitive classes closed under sequences of length less than $\lambda$ and containing the parameters.
\item[{\rm (b)}]
More generally, the category ${\bf Mod}\,T$ of models of a theory $T$
in $\Le_{\lambda}(\Sigma)$ is {\boldmath$\Delta_2$} 
with parameters $\{\lambda,\Sigma,T\}$, and it is absolute
between transitive classes closed under sequences of length less than $\lambda$ and containing the parameters.
\end{itemize}
\end{proposition}

\begin{proof}
In order to claim that $X$ is a $\Sigma$\nobreakdash-structure,
we need to formalize the following statement: ``$\lambda$ is a regular cardinal,
and $\Sigma=\langle S,\Sigma_{\rm op},\Sigma_{\rm rel},{\rm ar}\rangle$ is a $\lambda$\nobreakdash-ary signature,
and $X=\langle\{ X_s: s\in S\},\,\{ \sigma_X: \sigma\in\Sigma_{\rm op}\},\,\{ \rho_X:\rho\in\Sigma_{\rm op}\}\rangle$ is a $\Sigma$\nobreakdash-structure''.
Writing down that $\lambda$ is a regular cardinal is $\Pi_1$
by \cite[Lemma~13.13]{J2}, and adding that $\Sigma$
is a $\lambda$\nobreakdash-ary signature does not increase complexity. 
The assertion that $X$ is a $\Sigma$\nobreakdash-structure 
includes the $\Pi_1$ formula
\[
\begin{array}{c}
(\forall\sigma\in\Sigma_{\rm op})\,(\forall \alpha\in\lambda)\,(\forall x)\,
[[\mbox{$x$ is a function $\alpha\to\cup_{s\in S}\,X_s$}
\\ [0.1cm]
\wedge\, \mbox{${\rm ar}(\sigma)=(\langle s_i:i\in\alpha\rangle\to s)$} \,\wedge\,
(\forall i\in\alpha)\,x(i)\in X_{s_i}]\to \sigma_X(x)\in X_s].
\end{array}
\]
Hence, the whole statement is~$\Pi_1$. Similarly, the assertion
that $f\colon X\to Y$ is a homomorphism of $\Sigma$\nobreakdash-structures is~$\Pi_1$,
since we need to impose that $f(\sigma_X(x))=\sigma_Y(f(x))$ for all functions
$x\colon\alpha\to\cup_{s\in S}\,X_s$ with $x(i)\in X_{s_i}$ for all $i\in\alpha$, for each
operation symbol $\sigma$ of arity $\langle s_i:i\in\alpha\rangle\to s$. 
Stating that $f(x)\in\rho_Y$ for every $x\in\rho_X$ and each relation symbol $\rho$ 
does not require unbounded quantifiers.

If $\lambda=\omega$, then we can omit the clause ``$\lambda$ is a regular cardinal''
and there is only need to quantify over finite sequences
in $\cup_{s\in S}\,X_s$, which is~$\Delta_1$, as discussed in Example~\ref{example3}.

In order to state that $X$ is a model of a theory~$T$, we need to assert that
``$X$ is a $\lambda$\nobreakdash-ary $\Sigma$\nobreakdash-structure, and
$T$ is a set of sentences of the language of~$\Sigma$,
and every sentence of $T$ is satisfied in~$X$''. 
If $\lambda=\omega$, then this is again~$\Delta_1$, since satisfaction
of sentences of a finitary language in $X$ only depends on finite subsets of~$X$. 
For an arbitrary regular cardinal~$\lambda$, the last two clauses
are absolute between transitive classes that are closed under sequences
of length less than~$\lambda$.
Hence, by the Reflection Principle, $X$ is a model of $T$ if and only if every $\varphi\in T$ is a sentence of the language of~$\Sigma$, and $X$ is a $\Sigma$\nobreakdash-structure, 
and there is a finite fragment ${\rm ZFC}^*$ of ZFC such that
\begin{equation}
\label{trick1}
\begin{array}{c}
\exists M\, (\mbox{$M$ is transitive and closed under $<\!\lambda$-sequences} 
\\[0.1cm]
\wedge\,\{\lambda,\Sigma,T,X\}\subset M \,
\wedge\, M\models\textstyle\bigwedge{\rm ZFC}^*
\,\wedge\, M\models (\forall\varphi\in T)\, X\models \varphi),
\end{array}
\end{equation}
which can be replaced with
\begin{equation}
\label{trick2}
\begin{array}{c}
\forall M\, ((\mbox{$M$ is transitive and closed under $<\!\lambda$-sequences}
\\[0.1cm]
\wedge\,\{\lambda,\Sigma,T,X\}\subset M \,
\wedge\, M\models\textstyle\bigwedge{\rm ZFC}^* 
) \to M\models (\forall\varphi\in T)\, X\models \varphi).
\end{array}
\end{equation}
Since \eqref{trick1} is $\Sigma_2$ and \eqref{trick2} is~$\Pi_2$,
the statement ``$X$ is a model of~$T$'' is~$\Delta_2$. And a morphism
between models of $T$ is just a homomorphism of $\Sigma$\nobreakdash-struc\-tures,
so the proof is complete.
\end{proof}

\section{Supporting elementary embeddings}
\label{elementaryembeddings}

An \emph{elementary embedding} of a $\Sigma$\nobreakdash-structure $X$ into another 
$\Sigma$\nobreakdash-structure $Y$ (where $X$ and $Y$ can be proper classes)
is a function $j\colon X\to Y$ that preserves and reflects truth.
That is, for every formula $\varphi(x_i)_{i\in I}$ of the language of $\Sigma$
and all $\{ a_i: i\in I \}$ in~$X$, the sentence 
$\varphi(a_i)_{i\in I}$ is satisfied in $X$ if and only if $\varphi(j(a_i))_{i\in I}$ is satisfied in~$Y$.

In what follows, we consider elementary embeddings between structures of the language of set theory.
If $j\colon V\to M$ is a nontrivial elementary embedding
of the universe $V$ of all sets into a transitive class~$M$, then its \emph{critical point}
(i.e., the least ordinal moved by~$j$) is a measurable cardinal.
In fact, the existence of a nontrivial elementary embedding
of the set-theoretic universe into a transitive class is equivalent to the existence of 
a measurable cardinal \cite[Lemma~17.3]{J2}.

Given a subcategory $\Ce$ of the category of sets and an elementary embedding $j\colon V\to M$,
we say that $j$ is \emph{supported} by~$\Ce$ if, for every object $X$ in~$\Ce$,
the set $j(X)$ is also in $\Ce$ and the restriction function 
$j\restriction X : X\to j(X)$ is a morphism in~$\Ce$.

\begin{theorem}
\label{lemma0}
Let $j\colon V\to M$ be an elementary embedding with critical point~$\kappa$. 
Let $\Sigma$ be a $\lambda$\nobreakdash-ary signature in~$V_{\kappa}$ for
a regular cardinal $\lambda < \kappa$ such that 
$M$ is closed under sequences
of length less than~$\lambda$.
If $X$ is a $\Sigma$\nobreakdash-structure, then $j(X)$ is also
a $\Sigma$\nobreakdash-structure and $j\restriction X : X\to j(X)$ 
is an elementary embedding of $\Sigma$\nobreakdash-structures. 
\end{theorem}

\begin{proof}
First, observe that $j(\lambda)=\lambda$ and hence $\lambda$ is also a regular cardinal in~$M$. 
Next, $j(\Sigma)=\Sigma$ as $\Sigma\in V_{\kappa}$. Therefore, since $j$ is an elementary embedding, 
if $X$ is a $\Sigma$\nobreakdash-structure then $j(X)$ is a $\Sigma$\nobreakdash-structure in~$M$. 
It follows that $j(X)$ is also a $\Sigma$\nobreakdash-structure in~$V$, 
because, by Proposition~\ref{nonabsolute}, being a $\lambda$\nobreakdash-ary $\Sigma$\nobreakdash-structure is absolute for transitive classes
containing $\lambda$ and closed under sequences
of length less than~$\lambda$.
 
We next check, by induction on the complexity of formulas of~$\mathcal{L}_\lambda (\Sigma)$, 
that $j\restriction X$ is an elementary embedding of $\Sigma$\nobreakdash-structures.
For atomic formulas, let $\sigma \in \Sigma_{\rm op}$ be an operation symbol with arity 
$\langle s_i:i\in \alpha\rangle\to s$ where $\alpha<\lambda$, so $j(\alpha)=\alpha$. 
Thus, if $a_i\in X_{s_i}$ for all $i\in \alpha$, and $a\in X_s$, then, since $j$ is elementary,
$X\models (\sigma_X(a_i)_{i\in\alpha}=a)$ if and only if
\[
M\models \Big( j(X)\models (\sigma_{j(X)}(j(a_i))_{i\in\alpha}=j(a))\Big).
\]
Since the statement $j(X)\models (\sigma_{j(X)}(j(a_i))_{i\in\alpha}=j(a))$ 
is absolute for transitive classes, it holds in $M$ if and only if it holds in~$V$, as needed.
Relation symbols $\rho\in \Sigma_{\rm rel}$ are dealt with similarly,
and the cases of negation and conjunction are immediate. Thus, there only remains to consider existential formulas. 
If $X\models \exists x\, \varphi (x, a)$ for some $a\in X$, then there exists $b\in X$ such  that 
$X\models \varphi (b,a)$. By induction hypothesis, $j(X)\models \varphi (j(b),j(a))$; hence 
$j(X)\models \exists x\, \varphi (x,j(a))$. 
For the converse, observe first that, since $M$ is transitive and closed under sequences
of length less than~$\lambda$, 
satisfaction in $j(X)$ of formulas of $\mathcal{L}_\lambda (\Sigma)$ is absolute between $M$ and~$V$.
Hence, if $j(X)\models \exists x\, \varphi (x,j(a))$
for some $a\in X$, then  $M\models (j(X)\models \exists x\, \varphi (x,j(a)))$, and, 
by elementarity of~$j$, we conclude that $X\models \exists x\, \varphi(x,a)$.
\end{proof}

Since elementary embeddings of $\Sigma$\nobreakdash-structures are homomorphisms,
Theorem~\ref{lemma0} tells us that categories of structures support
elementary embeddings with sufficiently large critical point.
The following generalization of this fact is a 
more accurate restatement of \cite[Proposition~4.4]{BCM}.

\begin{theorem}
\label{lemma1}
Let $\Ce$ be a class of $\Sigma$\nobreakdash-structures for some 
$\lambda$\nobreakdash-ary signature~$\Sigma$, where $\lambda$ is a regular cardinal.
Suppose that $\Ce$ is {\boldmath$\Sigma_1$} with a set $p$ of parameters.
Let $j\colon V\to M$ be an elementary embedding with critical point~$\kappa>\lambda$
such that $M$ is closed under sequences
of length less than~$\lambda$ and $\{p,\Sigma\}\in V_{\kappa}$.
If $X\in\Ce$, then $j(X)\in\Ce$ and $j\restriction X : X\to j(X)$ is 
an elementary embedding of $\Sigma$\nobreakdash-structures. 
\end{theorem}

\begin{proof}
The proof follows the same steps as the proof of Theorem~\ref{lemma0},
using the fact that $\Sigma_1$ formulas are upward absolute to infer
that $j(X)\in\Ce$ for every $X\in\Ce$.
\end{proof}

\section{Vop\v{e}nka's principle and supercompact cardinals}
\label{VPandsupercompactcardinals}

For any two structures $M\subseteq N$ of the language of set theory
and $n<\omega$, we write $M\preceq_n N$ and say that
$M$ is a \emph{$\Sigma_n$\nobreakdash-elementary substructure} of~$N$ if, for every $\Sigma_n$ formula 
$\varphi (x_1,\ldots ,x_k)$ and all $a_1,\ldots ,a_k \in M$,
\[
N\models \varphi (a_1,\ldots ,a_k)\; \mbox{ if and only if }\; M\models \varphi (a_1,\ldots ,a_k).
\]

For a cardinal $\lambda$, we denote by $H(\lambda)$ the set of all sets whose transitive closure 
has cardinality less than~$\lambda$. Thus $H(\lambda)$ is a transitive set contained
in~$V_{\lambda}$, and, if $\lambda$ is strongly inaccessible, then $H(\lambda)=V_{\lambda}$; 
see~\cite[Lemma~6.2]{Kun2}.

A~class $C$ of ordinals is \emph{unbounded} if it contains arbitrarily 
large ordinals, and it is \emph{closed} if, for every ordinal~$\alpha$, 
if $\bigcup(C\cap \alpha)=\alpha$ then $\alpha \in C$. The abbreviation \emph{club} means
closed and unbounded. As a consequence of the Reflection Principle \cite[Theorem~12.14]{J2},
for every $n$ there exists a club class of cardinals $\lambda$ such that $H(\lambda)\preceq_{n}V$.
In addition, if $\lambda$ is uncountable, then $H(\lambda)\preceq_{1}V$.

In what follows, structures are meant to be sets, not proper classes.
We say that $X$ and $Y$ are \emph{structures of the same type} if
they are both $\Sigma$\nobreakdash-structures for some signature~$\Sigma$.
\emph{Vop\v{e}nka's principle} is the following assertion; 
compare with \cite[Ch.\,6]{AR} or \cite[(20.29)]{J2}:

\medskip

\noindent
VP: 
\emph{For every proper class $\Ce$ of structures of the same type, there exist
distinct $X$ and $Y$ in $\Ce$ and an elementary embedding of $X$ into~$Y$}.

\medskip

This is a statement involving classes. In the language of set theory, 
one can also formulate~VP, but as an axiom schema, that is, an infinite set of axioms;
namely, one axiom for each formula $\varphi (x,y)$ of the language of set theory with two free variables,
as follows: 
\[
\begin{array}{ccc}
\forall x\, [(\forall y\, \forall z\, ((\varphi (x,y)\wedge \varphi(x,z))\to \mbox{$y$
and $z$ are structures of the same type}) \\[0.1cm]
\wedge\, \forall \alpha\, (\mbox{$\alpha$ is an ordinal}\to
\exists y\, (\text{rank}(y)>\alpha\, \wedge\, \varphi(x,y))))\to \\[0.1cm]
\exists y\,\exists z\,(\varphi(x,y)\,\wedge\, \varphi(x,z)\,\wedge\, y\ne z\, \wedge\, \exists e\,
(\mbox{$e\colon y\to z$ is elementary}))].
\end{array}
\]
In this article, VP will be understood as this axiom schema,
and similarly with the variants of VP defined below.

In the statement of VP, the requirement that there is an elementary embedding between two \emph{distinct} structures 
is sometimes replaced by the requirement that there is a \emph{nontrivial} elementary embedding
between two possibly equal structures. 
It follows from \cite{BT} that it is consistent with ZFC to assume that the two formulations
are equivalent. Equivalence can be proved
using rigid graphs, as in~\cite[\S 6.A]{AR}, although this seems to require the use of global choice.

The theory $\text{ZFC} + \text{VP}$ is very strong. It implies, for instance, that the class of 
extendible cardinals is stationary, that is, every club proper class contains 
an extendible cardinal~\cite{M}. The consistency of $\text{ZFC} + \text{VP}$ 
follows from that of ZFC plus the existence of an almost-huge cardinal; see \cite{J2} or~\cite{K}.

If $\lambda$ and $\nu$ are cardinals, we denote by $\nu^{<\lambda}$ the union of $\nu^{\alpha}$ for all $\alpha<\lambda$. If $f\colon A\to B$ is a homomorphism of structures and $M$ is any set, when we write that $f\in M$ we mean that $A,B\in M$ and $\{(a,f(a)):a\in A\}\in M$.

\begin{theorem}
\label{vpsigma1}
Let $\Ce$ be a full subcategory of $\Sigma$\nobreakdash-structures definable by a $\Sigma_1$ formula with a set $p$ of parameters for some 
$\lambda$\nobreakdash-ary signature~$\Sigma$. Let $\kappa$ be a regular cardinal bigger than $\lambda$ such that $\{p,\Sigma\}\in H(\kappa)$ and with the property that $\nu^{<\lambda} <\kappa$ for all $\nu <\kappa$. Then the following hold:
\begin{itemize}
\item[{\rm (a)}]
For every homomorphism $g\colon A\to Y$ of $\Sigma$\nobreakdash-structures with $A\in H(\kappa)$ and $Y\in\Ce$ there is a homomorphism $f\colon A\to X$ with $X\in\Ce\cap H(\kappa)$ and a commutative~triangle
\[
\xymatrix{ 
& A\ar[dl]_f\ar[dr]^g & \\
X\ar[rr]^e & & Y
}
\]
where $e$ is an elementary embedding.
\item[{\rm (b)}]
Every object $Y\in\Ce$ has a subobject $X\in\Ce\cap H(\kappa)$.
\end{itemize}
\end{theorem}

\begin{proof}
We only have to prove (a), since (b) then follows with $A=\emptyset$. Note that every elementary embedding of $\Sigma$\nobreakdash-structures is an injective homomorphism and, since $\Ce$ is a full subcategory, $e\colon X\to Y$ is in~$\Ce$, so $X$ is a subobject of $Y$, since, in a subcategory of sets, every injective morphism is a monomorphism; see~\cite[Proposition~7.37]{AHS}.

Thus, suppose that $\Ce$, viewed as a class, is definable as $\mathcal{C}=\{ x : \varphi (x, p)\}$, where $\varphi$ is $\Sigma_1$ 
and $p\in H(\kappa)$. 
Given $g\colon A\to Y$ with $A\in H(\kappa)$ and $Y\in\mathcal{C}$, let $\mu$ be a regular cardinal bigger than $\kappa$ such that $Y\in H(\mu)$ 
and such that $H(\mu)\models \varphi(Y,p)$. 

In this situation, the L\"owen\-heim--Skolem Theorem implies the existence of an elementary substructure $\langle N,\in\rangle$ 
of~$\langle H(\mu),\in\rangle$ 
of cardinality smaller than~$\kappa$ and  closed under sequences of length less than~$\lambda$ (here we use the assumption that $\nu^{<\lambda} <\kappa$ for all $\nu <\kappa$) such that $g\in N$ and with the transitive closure 
of $\{p,\Sigma,A\}$ contained in~$N$. 
By elementarity, $g$ is a homomorphism of $\Sigma$\nobreakdash-structures in $N$ and $N\models \varphi(Y,p)$.

Let $M$ be the transitive collapse of~$N$,
and let $j\colon M\to N$ be the isomorphism given by the collapse;
that is, $j$ is inverse to the function $\pi\colon N\to M$ given by $\pi(x)=\{\pi(z) : z\in x\}$;
see~\cite[6.13]{J2}. Since $N$ is closed under sequences of length less than~$\lambda$, so is $M$, and the critical point of $j$ is greater than or equal to $ \lambda$. And since $N$ contains the transitive closure of $\{p,\Sigma,A\}$, we have that $\pi(p)=p$, $\pi(\Sigma)=\Sigma$ and $\pi(A)=A$. Moreover, the restriction $j\restriction A$ is the identity.

Now let $X\in M$ be such that $j(X)=Y$ and let $f\colon A\to X$ be such that $j(f)=g$. Then $X\in H(\kappa)$ since $|M|<\kappa$ and $M$ is transitive. 
Since $j$ is an isomorphism and $j(p)=p$, we infer that $M\models \varphi (X,p)$,
and hence, as $\Sigma_1$ formulas are upward absolute for transitive classes, we  conclude that $X\in \mathcal{C}$ in~$V$. Since $j(\Sigma)=\Sigma$ and $M$ and $N$ are closed under sequences of length less than~$\lambda$, the object $X$ is a $\Sigma$\nobreakdash-structure and, since $j$ is an isomorphism,
the restriction $e=j\restriction X$ is an elementary embedding, hence a homomorphism of $\Sigma$\nobreakdash-structures.
Moreover, $f$ is also a homomorphism and the
triangle commutes since $f$ has been defined so that $g(a)=j(f(a))$ for all $a\in A$.
\end{proof}

Recall that a cardinal $\kappa$ is \emph{$\lambda$\nobreakdash-supercompact} if there is an 
elementary embedding $j\colon V\to M$ with $M$ transitive and 
with critical point~$\kappa$, such that $j(\kappa)>\lambda$ and $M$ is closed under $\lambda$\nobreakdash-sequences.
Note that it then follows that $H(\lambda)\in M$. 
A~cardinal $\kappa$ is called \emph{supercompact} if it is $\lambda$\nobreakdash-supercompact for all ordinals~$\lambda$.

The following theorem is an upgraded version of \cite[Theorem~4.5]{BCM},
where a similar result was proved for absolute classes.

\begin{theorem}
\label{theorem1}
Let $\Ce$ be a full subcategory of $\Sigma$\nobreakdash-structures definable by a $\Sigma_2$ formula with a set $p$ of parameters.
Suppose that there exists a supercompact cardinal~$\kappa$ bigger than 
the rank of $p$ and~$\Sigma$.
Then the following hold:
\begin{itemize}
\item[{\rm (a)}]
For every homomorphism $g\colon A\to Y$ of $\Sigma$\nobreakdash-structures with $A\in V_{\kappa}$ and $Y\in\Ce$ there is a homomorphism $f\colon A\to X$ with $X\in\Ce\cap V_{\kappa}$ and an elementary embedding $e\colon X\to Y$ with $e\circ f=g$. 
\item[{\rm (b)}]
Every object $Y\in\Ce$ has a subobject $X\in\Ce\cap V_{\kappa}$.
\end{itemize}
\end{theorem}

\begin{proof}
As with Theorem~\ref{vpsigma1}, we only have to prove~(a), since (b) follows by taking $A=\emptyset$.
Suppose that $\kappa$ is a supercompact cardinal for which $\{p,\Sigma,A\}\in V_{\kappa}$.
Then, since $\kappa$ is strongly inaccessible, we have $V_{\kappa}=H(\kappa)$ and, since $\kappa$ is regular, it is bigger than the supremum of the ordinals of the arities of all the operation symbols and relation symbols of $\Sigma$, so $\Sigma$ is $\kappa$\nobreakdash-ary.

Given a homomorphism $g\colon A\to Y$ with $Y\in \Ce$, let $\mu$ be a cardinal bigger than $\kappa$ such that $Y\in H(\mu)$ and
$H(\mu)\preceq_{2}V$. Let $j\colon V\to M$ be an elementary embedding with $M$ transitive and  
critical point~$\kappa$, such that $j(\kappa)>\mu$ and $M$ is closed under $\mu$\nobreakdash-sequences.
Then $j(A)=A$ since $A$ is in $H(\kappa)$, and $g$ and the restriction $j\restriction Y:Y\to j(Y)$ are in~$M$ because $A,Y\in M$ and $M$ is closed under $\mu$\nobreakdash-sequences. In addition, $g\colon A\to Y$ is a homomorphism of $\Sigma$\nobreakdash-structures in~$M$, since, by Proposition~\ref{nonabsolute}, 
being a homomorphism of $\kappa$\nobreakdash-ary $\Sigma$\nobreakdash-structures is absolute for transitive classes
containing $\Sigma$ and closed under sequences of length less than~$\kappa$.
Moreover, by Theorem~\ref{lemma0}, since $\Sigma\in V_{\kappa}$, the restriction $j\restriction Y:Y\to j(Y)$ is an elementary
embedding of $\Sigma$\nobreakdash-structures.

Since being a cardinal is $\Pi_1$ and hence downward absolute,
$\mu$ is a cardinal in~$M$, and this implies that $H(\mu)$ 
in the sense of $M$ coincides with $H(\mu)$. It follows that $H(\mu)\preceq_{1}M$, since
every $\Sigma_1$ sentence $\psi$ which holds in $M$ also holds in $V$ (as $\Sigma_1$
sentences are upward absolute) and therefore $\psi$ holds in $H(\mu)$ because
$H(\mu)\preceq_{2}V$. Hence, $\Sigma_2$ formulas are upward absolute between $H(\mu)$~and~$M$.
Since $H(\mu)\preceq_{2}V$ and the class $\Ce$ is defined by a $\Sigma_2$ formula~$\varphi(x,y)$, 
we have that $H(\mu)\models \varphi(Y,p)$ and thus $M\models\varphi(Y,p)$.

Now $\text{rank}(Y)<\mu<j(\kappa)$ in $V$ and also in~$M$.
Thus, as witnessed by $g\colon A\to Y$, in $M$ there exists a homomorphism $f\colon A\to X$ of $\Sigma$\nobreakdash-structures such that $\text{rank}(X)<j(\kappa)$
and $\varphi(X,p)$ holds, and there is an elementary embedding $e\colon X\to j(Y)$ such that $e\circ f=j(g)$.
By elementarity of~$j$, the corresponding statement is true in~$V$; that is, 
there exists a homomorphism of $\Sigma$\nobreakdash-structures $f\colon A\to X$ such that $\text{rank}(X)<\kappa$ and $\varphi(X,p)$ holds, so $X\in\Ce$,
and there is an elementary embedding $e\colon X\to Y$ with $e\circ f=g$, as we wanted to prove.
\end{proof}

Theorem~\ref{theorem1} tells us that the existence
of arbitrarily large supercompact cardinals implies that
VP holds for {\boldmath$\Sigma_2$} proper classes.
The following theorem yields a strong converse of this fact.

\begin{theorem}
\label{proposition000}
Suppose that, for every {\boldmath$\Delta_2$} proper class
$\Ce$ of structures in the language of set theory with one additional constant symbol,
there exist distinct $X$ and $Y$ in~$\Ce$ and an elementary embedding of $X$ into~$Y$.
Then there exists a proper class of supercompact cardinals.
\end{theorem}

\begin{proof}
Let $\xi$ be any ordinal and
suppose, towards a contradiction, that there are no supercompact cardinals bigger than~$\xi$.
Then the class function $F$ given as follows is well defined on ordinals $\zeta > \xi$:
$F(\zeta)$ equals the least cardinal $\lambda > \zeta$ 
such that no cardinal $\kappa$ such that $\xi <\kappa \leq \zeta$ is $\lambda$\nobreakdash-supercompact.
Since the assertion ``$\zeta$ is $\lambda$\nobreakdash-supercompact'' is $\Delta_2$ in ZFC (see~\cite[\S 22]{K}), 
$F$ is {\boldmath$\Delta_2$}\nobreakdash-definable with $\xi$ as a parameter.
Let 
\[
C_0=\{ \alpha : \mbox{$\alpha$ is a limit ordinal, $\xi<\alpha$, and $\forall \zeta \, ( \xi < \zeta <\alpha 
\to F(\zeta)<\alpha )$}\}.
\] 
Then $C_0$ is a club class {\boldmath$\Delta_2$}\nobreakdash-definable with $\xi$ as a parameter.

Fix a rigid binary relation (i.e., a rigid graph) $R$ on $\xi +1$ (see~\cite{Ne}). 
For each ordinal~$\alpha$, let $\lambda_\alpha$ be the least element of $C_0$ greater than~$\lambda$. 
The proper class $\Ce=\{ \langle V_{\lambda_\alpha +2},\in , \langle \alpha , R\rangle \rangle : \alpha >\xi\}$ 
is {\boldmath$\Delta_2$}-definable with $R$ as a parameter.
By our assumption, there exist $\alpha < \beta$ greater than $\xi$ and an elementary embedding 
\[
j : \langle V_{\lambda_\alpha +2},\in , \langle \alpha , R \rangle \rangle \longrightarrow 
\langle V_{\lambda_\beta +2}, \in , \langle \beta ,R\rangle \rangle .
\] 
Since $j$ must send $\alpha$ to~$\beta$, it is not the identity. Hence, by Kunen's Theorem (\cite[Theorem~17.7]{J2}, \cite{Kun1}), 
we have $\lambda_\alpha <\lambda_\beta$. Let $\kappa\leq \alpha$ be the critical point of~$j$. 
Then, as in \cite[Lemma~2]{M}, it follows that $\kappa$ is $\lambda_\alpha$\nobreakdash-supercompact. 
But this is impossible, since $F(\kappa) < \lambda_\alpha$ because $\lambda_\alpha \in C_0$.
\end{proof}

In order to summarize what we have proved so far,
we introduce some useful notation. Let $\Gamma$ be one of $\Sigma_n$, $\Pi_n$, $\Delta_n$, $\Sigma_n\wedge\Pi_n$ or
{\boldmath$\Sigma_n$}, {\boldmath$\Pi_n$}, {\boldmath$\Delta_n$}, {\boldmath$\Sigma_n\wedge\Pi_n$},
for any~$n$. 
For an infinite cardinal $\kappa$ and a signature 
$\Sigma\in H(\kappa)$, we write:

\bigskip

\noindent
${\rm VP}^{\Sigma}(\Gamma)$:
\emph{For every $\Gamma$ proper class $\Ce$ of $\Sigma$\nobreakdash-structures, 
there exist distinct $X$ and $Y$ in $\Ce$ and an elementary embedding of $X$ into~$Y$}.

\bigskip

\noindent
${\rm SVP}_{\kappa}^{\Sigma}(\Gamma)$:
\emph{For every proper class $\Ce$ of $\Sigma$\nobreakdash-structures admitting
a $\Gamma$ definition whose parameters, if any, are in~$H(\kappa)$, and for every $Y\in \Ce$, 
there exists $X\in \Ce \cap H(\kappa)$ and an elementary embedding of $X$ into~$Y$}.

\bigskip

If $\Sigma$ is omitted from the notation, we mean
that the corresponding statement holds for all admissible signatures. Thus, ${\rm VP}(\Gamma)$
means ${\rm VP}^{\Sigma}(\Gamma)$ for all~$\Sigma$, while ${\rm SVP}_{\kappa}(\Gamma)$ means
${\rm SVP}_{\kappa}^{\Sigma}(\Gamma)$ for every $\Sigma\in H(\kappa)$.

Even though ${\rm SVP}_{\kappa}^{\Sigma}(\Gamma)$ is an apparently  
stronger statement than ${\rm VP}^{\Sigma}(\Gamma)$ (hence the notation SVP), 
in the case of {\boldmath$\Sigma_2$} classes of structures they turn out  
to be equivalent, as we next prove.

\begin{corollary}
\label{corollary0}
The following statements are equivalent:
\begin{enumerate}
\item ${\rm SVP}_{\kappa}(\mathbf{\Sigma_2})$ holds for a proper class of cardinals~$\kappa$.
\item ${\rm VP}(\mathbf{\Sigma_2})$ holds.
\item ${\rm VP}^{\Sigma}(\mathbf{\Delta_2})$ holds if $\Sigma$ is
the signature of the language of set theory with one additional constant symbol.
\item There exists a proper class of supercompact cardinals.
\end{enumerate}
\end{corollary}

\begin{proof}
In order to check that (1)~$\Rightarrow$~(2), suppose that (1) is true,
and let $\Sigma$ be any signature. Let $\Ce$ be any proper class of $\Sigma$\nobreakdash-structures
defined by a $\Sigma_2$ formula with parameters, and let $\kappa$ be bigger than the ranks
of the parameters and such that ${\rm SVP}_{\kappa}^{\Sigma}(\mathbf{\Sigma_2})$ holds.
Since $\Ce$ is a proper class, we may choose $Y$ of rank bigger than~$\kappa$,
so any $X\in\Ce\cap H(\kappa)$ will necessarily be distinct from~$Y$.
Hence, there exist distinct $X$ and $Y$ such that $X$ is elementarily embeddable into~$Y$,
so ${\rm VP}^{\Sigma}(\mathbf{\Sigma_2})$ holds, as needed.
The implication (2)~$\Rightarrow$~(3) is trivial, and
Theorem~\ref{proposition000} implies that (3)~$\Rightarrow$~(4).
Finally, to see that (4)~$\Rightarrow$~(1), let $\xi$ be any cardinal and
pick a supercompact cardinal $\kappa > \xi$.
Since $H(\kappa) = V_{\kappa}$, Theorem~\ref{theorem1} tells us
that ${\rm SVP}_{\kappa}(\mathbf{\Sigma_2})$ holds.
\end{proof}

The following is a corresponding version without parameters,
with the same (in fact, simpler) proof.

\begin{corollary}
\label{corollary-1}
The following statements are equivalent:
\begin{enumerate}
\item ${\rm SVP}_{\kappa}(\Sigma_2)$ holds for some cardinal~$\kappa$.
\item ${\rm VP}(\Sigma_2)$ holds.
\item ${\rm VP}^{\Sigma}(\Delta_2)$ holds if $\Sigma$ is
the signature of the language of set theory.
\item There exists a supercompact cardinal.
\end{enumerate}
\end{corollary}

\section{Vop\v{e}nka's principle and extendible cardinals}
\label{VPandextendiblecardinals}

For cardinals $\kappa <\lambda$, we say that $\kappa$ is \emph{$\lambda$\nobreakdash-extendible} 
if there is an elementary embedding $j\colon V_{\lambda}\to V_{\mu}$ for some~$\mu$, 
with critical point $\kappa$ and with $j(\kappa)>\lambda$.
A cardinal $\kappa$ is called \emph{extendible} if it is $\lambda$\nobreakdash-extendible 
for all cardinals $\lambda >\kappa$. As shown in
\cite[20.24]{J2}, extendible cardinals are supercompact.
See \cite{J2} or \cite{K} for more information about extendible cardinals.

For each $n<\omega$, let $\Cn$ denote the club proper class of infinite cardinals 
$\kappa$ that are \emph{$\Sigma_n$\nobreakdash-correct} in~$V$, that is, $V_{\kappa}\preceq_n V$.
Since the satisfaction relation $\models_n$ for $\Sigma_n$ sentences
(which is, in fact, a proper class)
is $\Sigma_n$\nobreakdash-definable
for $n\geq 1$ \cite[\S 0.2]{K}, it follows that, for $n\geq 1$, the class $\Cn$ is~$\Pi_{n}$. 
To see this, note first that $\Czero$ is the class of all infinite cardinals, 
and therefore it is $\Pi_1$\nobreakdash-definable. For $\kappa$ an infinite cardinal, $\kappa \in \Cone$ 
if and only if $\kappa$ is an uncountable cardinal and $V_{\kappa}=H(\kappa)$, 
which implies that $\Cone$ is $\Pi_1$\nobreakdash-definable.
In general, for $n\geq 1$ and for any infinite cardinal~$\kappa$, we have
$V_{\kappa}\preceq_{n+1} V$  if and only if  
\[
\kappa\in \Cn\,\wedge \, (\forall \varphi (x)\in 
\Sigma_{n+1})\, (\forall a\in V_{\kappa})\, (\;\models_{n+1}\varphi (a) \to V_{\kappa} \models \varphi (a)),
\]
which is a $\Pi_{n+1}$ formula
showing that $\Cnpo$ is $\Pi_{n+1}$\nobreakdash-definable.

We shall use the following new strong form of extendibility.

\begin{definition}
\label{definition1}
{\rm 
For $C$ a club proper class of cardinals and $\kappa <\lambda$ in~$C$, we say that $\kappa$ is 
\emph{$\lambda$\nobreakdash-$C$\nobreakdash-extendible} if there is an elementary embedding 
\hbox{$j\colon V_{\lambda}\to V_{\mu}$} for some $\mu\in C$, with critical point~$\kappa$, 
such that $j(\kappa)>\lambda$ and $j(\kappa)\in C$.

We say that a cardinal $\kappa$ in $C$ is \emph{$C$\nobreakdash-extendible} if it is 
$\lambda$\nobreakdash-$C$\nobreakdash-extendible for all $\lambda$ in $C$ greater than~$\kappa$.
}
\end{definition}

Note that, for all~$n$, if $\kappa$ is $\Cn$\nobreakdash-extendible, then $\kappa$ is extendible.
Therefore, a cardinal is $\Czero$\nobreakdash-extendible if and only if it is extendible.

\begin{proposition}
\label{proposition3}
Every extendible cardinal is $\Cone$\nobreakdash-extendible.
\end{proposition}

\begin{proof}
Suppose that $\kappa$ is extendible and $\lambda \in \Cone$ is greater than~$\kappa$. 
Note that the existence of an extendible cardinal implies the existence of a proper class of 
inaccessible cardinals, as the image of $\kappa$ under any elementary embedding 
\hbox{$j\colon V_{\lambda}\to V_{\mu}$}, with critical point $\kappa$ and $\lambda$ a cardinal, 
is always an inaccessible cardinal in~$V$. So we can pick an inaccessible cardinal $\lambda'\geq \lambda$. 
Let \hbox{$j'\colon V_{\lambda'}\to V_{\mu'}$} be an elementary embedding with critical point $\kappa$ and such 
that \hbox{$j'(\kappa )>\lambda'$}. Since $V_{\lambda'}=H(\lambda')$, it follows by elementarity of $j'$ 
that $V_{\mu'}=H(\mu')$. Hence, $\mu'\in \Cone$. 

Let us see that $j=j'\restriction V_{\lambda}:V_{\lambda}\to V_{j'(\lambda)}$ witnesses the 
$\lambda$\nobreakdash-$\Cone$\nobreakdash-extendibility of~$\kappa$. We only need to check that $\mu=j'(\lambda)\in \Cone$. 
But since $V_{\lambda}\preceq_1 V_{\lambda'}$, it follows by elementarity of $j'$ that 
$V_{\mu}\preceq_1 V_{\mu'}$. Hence, since $\mu'\in \Cone$, also $\mu \in \Cone$.
\end{proof}

Hence, a cardinal is $\Cone$\nobreakdash-extendible if and only if it is extendible.
Let us also observe that, if there exists a $\Cnpt$\nobreakdash-extendible cardinal for $n\ge 1$, 
then there exists a proper class of $\Cn$\nobreakdash-extendible cardinals; see~\cite{Ba}.

\begin{lemma}
\label{extendiblecorrectness}
If $\kappa$ is $C(n)$\nobreakdash-extendible, then $\kappa\in C(n+2)$.
\end{lemma}

\begin{proof}
By induction on $n$. For $n=0$, since $\kappa\in\Cone$, we only need to show that if 
$\exists x\, \varphi (x)$ is a $\Sigma_2$ sentence, where $\varphi$ is $\Pi_1$ and has parameters in~$V_{\kappa}$, 
that holds in~$V$, then it holds in~$V_{\kappa}$. So suppose that $a$ is such that $\varphi(a)$ holds in~$V$. 
Let $\lambda\in C(n)$ be greater than $\kappa$ and with $a\in V_{\lambda}$, and let $j\colon V_{\lambda}\to V_{\mu}$ 
be elementary, with critical point $\kappa$ and with $j(\kappa)>\lambda$. Then $V_{j(\kappa)}\models \varphi (a)$, 
and so, by elementarity, $V_{\kappa}\models \exists x\,\varphi (x)$.

Now suppose that $\kappa$ is $C(n)$\nobreakdash-extendible and $\exists x\, \varphi (x)$ is a $\Sigma_{n+2}$ sentence, 
where $\varphi$ is $\Pi_{n+1}$ and has parameters in~$V_{\kappa}$. If $\exists x\,\varphi(x)$ holds in~$V_{\kappa}$, 
then, since by the induction hypothesis $\kappa \in C(n+1)$, we have that $\exists x\,\varphi(x)$ holds in~$V$.
Now suppose that $a$ is such that $\varphi(a)$ holds in~$V$. Let $\lambda \in C(n)$ be greater than $\kappa$ and 
such that $a\in V_{\lambda}$, and let $j\colon V_{\lambda}\to V_{\mu}$ be elementary with critical point $\kappa$ 
and with $j(\kappa)>\lambda$. Then, since $j(\kappa)\in C(n)$, we have $V_{j(\kappa)}\models \varphi (a)$, and so, 
by elementarity, $V_{\kappa}\models \exists x\,\varphi (x)$.
\end{proof}

\begin{theorem}
\label{theorem4}
For every $n\ge 1$, if  $\kappa$ is a $\Cn$\nobreakdash-extendible cardinal, then 
${\rm SVP}_{\kappa}(\mbox{\boldmath$\Sigma_{n+2}$})$ holds.
\end{theorem}

\begin{proof}
Fix a $\Sigma_{n+2}$ formula $\exists x\, \varphi (x,y,z)$, where $\varphi$ is $\Pi_{n+1}$, such that
\[
\Ce =\{ Y : \exists x\, \varphi (x,Y,p)\}
\]
is a proper class of structures of the same type
for some set $p\in V_{\kappa}$.

Fix $Y\in \Ce$ and let $\lambda \in C(n+2)$ be greater than $\kappa$ and the ranks of $p$ and~$Y$. 
Thus, 
$
V_{\lambda}\models \exists x\, \varphi (x,B,p)
$.
Let $j\colon V_{\lambda}\to V_{\mu}$ for some $\mu\in \Cn$ be an elementary embedding with critical 
point~$\kappa$, with $j(\kappa)>\lambda$ and $j(\kappa)\in \Cn$. Note that both $Y$ and 
$j\restriction Y\colon Y\to j(Y)$ are in~$V_{\mu}$. 

Since $\kappa,\lambda\in C(n+2)$ by Lemma~\ref{extendiblecorrectness}, and $\kappa <\lambda$, we have $V_{\kappa}\preceq_{n+2}V_{\lambda}$. 
It follows that $V_{j(\kappa)}\preceq_{n+2} V_{\mu}$. Indeed, the following holds:
\[
V_{\lambda}\models (\forall x\in V_{\kappa})\, (\forall \theta \in \Sigma_{n+2}) \, (V_{\kappa}\models \theta (x) 
\leftrightarrow \;\models_{n+2} \theta (x)).
\]
Hence, by elementarity,
\[
V_{\mu}\models (\forall x\in V_{j(\kappa)})\, (\forall \theta \in \Sigma_{n+2}) \, 
(V_{j(\kappa)}\models \theta (x) \leftrightarrow \;\models_{n+2} \theta (x)),
\]
which implies that $V_{j(\kappa)}\preceq_{n+2} V_{\mu}$.

Since $j(\kappa)\in \Cn$, we  have $V_{\lambda}\preceq_{n+1} V_{j(\kappa)}$, and  
therefore $V_{\lambda}\preceq_{n+1} V_{\mu}$. It~follows that $V_{\mu}\models \exists x\,\varphi (x,Y,b)$. 

Thus, in $V_{\mu}$ it is true that there exists $X\in V_{j(\kappa)}$ such that $X\in \Ce$, namely~$Y$, 
and there exists an elementary embedding $e\colon X\to j(Y)$, namely $j\restriction Y$. 
Therefore, by elementarity of~$j$, the same is true in~$V_{\lambda}$, that is, 
there exists $X\in V_{\kappa}$ such that $X\in \Ce$, and there exists an elementary embedding 
$e\colon X\to Y$. 
Since $\lambda\in C(n+2)$, we have $X\in \Ce$ and we are done.
\end{proof}

\begin{corollary}
If $\kappa$ is an extendible cardinal, then ${\rm SVP}_{\kappa}(\mathbf{\Sigma_{3}})$ holds.
\end{corollary}

\begin{proof}
This is the assertion of Theorem~\ref{theorem4} for $n=1$.
\end{proof}

\begin{corollary}
\label{theorem5}
Let $\Ce$ be a full subcategory of $\Sigma$\nobreakdash-structures definable by a $\Sigma_{n+2}$
formula with a set $p$ of parameters, where $n\ge 1$.
Suppose that there exists a $C(n)$\nobreakdash-extendible cardinal~$\kappa$ bigger than 
the rank of $p$ and~$\Sigma$.
Then the following hold:
\begin{itemize}
\item[{\rm (a)}]
For every homomorphism $g\colon A\to Y$ of $\Sigma$\nobreakdash-structures with $A\in V_{\kappa}$ and $Y\in\Ce$ there is a homomorphism $f\colon A\to X$ with $X\in\Ce\cap V_{\kappa}$ and an elementary embedding $e\colon X\to Y$ with $e\circ f=g$.
\item[{\rm (b)}]
Every object $Y\in\Ce$ has a subobject $X\in\Ce\cap V_{\kappa}$.
\end{itemize}
\end{corollary}

\begin{proof}
Part~(b) is a consequence of Theorem~\ref{theorem4} and part~(a) is a more general variant proved as in Theorem~\ref{theorem1}.
\end{proof}

The following theorem yields a converse to Theorem~\ref{theorem4}. 

\begin{theorem}
\label{theorem10}
Let $n\geq 1$, and suppose that ${\rm VP}^{\Sigma}(\Sigma_{n+1}\wedge \Pi_{n+1})$ holds
when $\Sigma$ is the signature of the language of set theory with finitely many
additional $1$\nobreakdash-ary relation symbols.
Then there exists a $\Cn$\nobreakdash-extendible cardinal.
\end{theorem}

\begin{proof}
Suppose, to the contrary, that there is no $\Cn$\nobreakdash-extendible cardinal. 
Then the class function $F$ on ordinals given by defining
$F(\zeta)$ to be the least $\lambda >\zeta$ such that $\lambda\in \Cn$ and $\zeta$ is not 
$\lambda$\nobreakdash-$\Cn$\nobreakdash-extendible is well defined.

For $\lambda\in \Cn$, the relation ``$\zeta$ is $\lambda$\nobreakdash-$\Cn$\nobreakdash-extendible'' 
is $\Sigma_{n+1}$, for it holds if and only if $\zeta\in \Cn$ and  
\[
\exists \mu \, \exists j\colon V_{\lambda}\to V_{\mu} \,(\mbox{$j$ is elementary}\, \wedge \, \mbox{cp}(j)=\zeta \,
 \wedge \, j(\zeta)>\lambda \, \wedge \, \mu, j(\zeta)\in \Cn ),
\]
where $\mbox{cp}(j)$ denotes the critical point of~$j$.
Hence $F$ is $\Sigma_{n+1}\wedge \Pi_{n+1}$.

Let $C=\{ \alpha : \mbox{$\alpha$ is a limit ordinal and } (\forall \zeta <\alpha) \, F(\zeta )<\alpha \}$. 
So, $C$ is a $\Sigma_{n+1}\wedge \Pi_{n+1}$ closed unbounded proper class.

For each ordinal $\alpha$, let $\lambda_{\alpha}$ be the first limit point of 
$D=C\cap \Cn$ above~$\alpha$. Note that the 
class function $f$ on ordinals such that $f(\alpha)=\lambda_\alpha$ is $(\Sigma_{n+1}\wedge \Pi_{n+1})$\nobreakdash-definable. 
Now let
\[
\Ce =\{\langle V_{\lambda_{\alpha}},\in , \alpha  ,\lambda_\alpha , C\cap \alpha +1 \rangle : \alpha \in D\}.
\]
We claim that $\mathcal{C}$ is $(\Sigma_{n+1} \wedge \Pi_{n+1})$\nobreakdash-definable. 
Indeed, $X\in \mathcal{C}$ if and only if $X=\langle X_0, X_1, X_2 ,X_3,X_4 \rangle$, where
\[
\begin{array}{lll}
(1)\; X_2 \in C; \qquad &
(2)\; X_3 =\lambda_{X_2}; \qquad &
(3)\; X_0=V_{X_3};
\\[0.2cm]
(4)\; X_1=\,\in \restriction X_0; \qquad &
(5)\; X_4=C\cap X_2 +1. & 
\end{array}
\]

We have already seen that (1) and (2) are $\Sigma_{n+1} \wedge \Pi_{n+1}$ expressible. 
And so are (3) and~(4). As for~(5), note that $X_4=C\cap \alpha +1$ holds in $V$ if and only if 
it holds in $V_{X_3}$. 

So $\Ce$ is a $\Sigma_{n+1}\wedge \Pi_{n+1}$ proper class of structures of the same type in the 
language of set theory with three additional relation symbols. By our assumption, there are $\alpha <\beta$ 
in $D$ and an elementary embedding
\[
j\colon\langle V_{\lambda_{\alpha}},\in,  \alpha  ,\lambda_\alpha , C\cap \alpha +1 \rangle 
\longrightarrow
\langle V_{\lambda_{\beta}},\in,  \beta  ,\lambda_\beta , C\cap \beta +1 \rangle.
\]
Since $j$ sends $\alpha$ to $\beta$, it is not the identity. Let $\kappa$ be the critical point of~$j$.

Since $\alpha \in C$, we have $\kappa < F(\kappa )<\alpha$. Thus,
\[
j\restriction V_{F(\kappa)}:V_{F(\kappa)}\longrightarrow V_{j(F(\kappa))}
\]
is elementary, with critical point $\kappa$. 

We claim that $\kappa \in D$. Otherwise, $\gamma =\mbox{sup}(D\cap \kappa )<\kappa$. 
Let $\delta$ be the least ordinal in $D$ greater than $\gamma$ with $\kappa <\delta <\lambda_{\alpha}$.
Since $\delta$ is definable from $\gamma$ in the structure $\langle V_{\lambda_{\alpha}},\in , \alpha  ,C\cap \alpha +1  \rangle$, and since $j(\gamma)=\gamma$, we must also have $j(\delta)=\delta$. 
But then $j\restriction V_{\delta +2}:V_{\delta +2}\to V_{\delta +2}$ is an elementary embedding, 
contradicting Kunen's Theorem \cite{Kun1}.

By elementarity, $j(\kappa )\in \Cn$. Moreover, since $F(\kappa )\in \Cn$ and 
$\lambda_{\beta}\in \Cn$, we have $j(F(\kappa))\in \Cn$.
Since $\kappa \in C$, by elementarity we also have $j(\kappa )\in C$. Hence, $j(\kappa)>F(\kappa)$.
This shows that $j\restriction V_{F(\kappa)}$ witnesses that $\kappa$ is $F(\kappa)$\nobreakdash-$\Cn$\nobreakdash-extendible,
and this contradicts the definition of~$F$.
\end{proof}

The proof of Theorem~\ref{theorem10}
easily generalizes to the boldface case (see the proof of Theorem~\ref{proposition000}), namely if 
${\rm VP}(\mbox{\boldmath$\Sigma_{n+1}\wedge\Pi_{n+1}$})$
holds, then there is a proper class of
$\Cn$\nobreakdash-exten\-dible cardinals. In fact it is sufficient to assume that
${\rm VP}^{\Sigma}(\mbox{\boldmath$\Sigma_{n+1}\wedge\Pi_{n+1}$})$
holds when $\Sigma$
is the signature of the language of set theory with a finite number of additional $1$\nobreakdash-ary relation symbols.

The following corollaries summarize our results in this section.

\begin{corollary}
\label{corollary2}
The following statements are equivalent for $n\geq 1$:
\begin{enumerate}
\item 
${\rm SVP}_{\kappa}(\mbox{\boldmath$\Sigma_{n+2}$})$
holds for some cardinal~$\kappa$.
\item 
${\rm VP}(\Sigma_{n+1}\wedge \Pi_{n+1})$ holds.
\item 
${\rm VP}^{\Sigma}(\Sigma_{n+1}\wedge \Pi_{n+1})$ 
holds when $\Sigma$
is the signature of the language of set theory with a finite number of additional $1$\nobreakdash-ary relation symbols.
\item There exists a $\Cn$\nobreakdash-extendible cardinal.
\end{enumerate}
\end{corollary}

\begin{corollary}
\label{corollary2bis}
The following statements are equivalent:
\begin{enumerate}
\item For every~$n$, 
${\rm SVP}_{\kappa}(\mbox{\boldmath$\Sigma_{n}$})$
holds for a proper class of cardinals~$\kappa$.
\item For every~$n$, 
${\rm SVP}_{\kappa}(\mbox{\boldmath$\Sigma_{n}$})$
holds for some cardinal~$\kappa$.
\item ${\rm VP}(\mbox{\boldmath$\Sigma_n$})$ holds for all~$n$.
\item ${\rm VP}^{\Sigma}(\Sigma_n)$ holds for all $n$ when $\Sigma$
is the signature of the language of set theory with a finite number of additional $1$\nobreakdash-ary relation symbols.
\item There exists a $\Cn$\nobreakdash-extendible cardinal for every~$n$.
\item There exists a proper class of $\Cn$\nobreakdash-extendible cardinals for every~$n$.
\item Vop\v{e}nka's principle holds.
\end{enumerate}
\end{corollary}

\section{Accessible categories}
\label{accessiblecategories}

A category is \emph{small} if its objects form a set,
and \emph{essentially small} if the isomorphism classes of its objects form a set.

Let $\lambda$ be a regular cardinal.
A~nonempty category $\Ke$ is called \emph{$\lambda$\nobreakdash-filtered}
if, given any set of objects $\{k_i\}_{i\in I}$ in $\Ke$ where $|I|<\lambda$, there is an object 
$k\in\Ke$ and a morphism $k_i\to k$ for each $i\in I$, and, moreover, given any set of parallel arrows
between any two objects $\{f_j\colon k\to k'\}_{j\in J}$ where $|J|<\lambda$, there is a morphism
$g\colon k'\to k''$ such that $g\circ f_j$ is the same morphism for all $j\in J$.
If $\Ce$ is any category, a functor $D\colon \Ke\to\Ce$ where $\Ke$ is a $\lambda$\nobreakdash-filtered
small category is called a \emph{$\lambda$\nobreakdash-filtered diagram}, and, if $D$ has a colimit~$L$, then 
$L$ is called a \emph{$\lambda$\nobreakdash-filtered colimit}. 
For example, every set is a $\lambda$\nobreakdash-filtered colimit 
of its subsets of cardinality smaller than~$\lambda$ (partially ordered by inclusion).

An object $A$ of a category $\Ce$ is \emph{$\lambda$\nobreakdash-presentable}
if the functor $\Ce(A,-)$ preserves $\lambda$\nobreakdash-filtered colimits; that is,
for each $\lambda$\nobreakdash-filtered diagram $D\colon \Ke\to\Ce$ with a colimit~$L$,
each morphism $A\to L$ factors through a morphism $A\to Dk$ for some $k\in\Ke$,
and if two morphisms $A\to Dk$ and $A\to Dk'$ compose
to the same morphism $A\to L$, then there is some $k''\in\Ke$ and morphisms $k\to k''$
and $k'\to k''$ in $\Ke$ such that the two composites $A\to Dk''$ are equal;
see \cite[\S 6.1]{GU} or~\cite[\S 2.1]{MP}.

For a small full subcategory $\Ae$ of $\Ce$ and an 
object $X$ in~$\Ce$, the \emph{canonical diagram} $(\Ae\downarrow X)\to \Ce$
sends each pair $\langle A,f\rangle$ with $f\in\Ce(A,X)$ to~$A$.
Recall from \cite[1.23]{AR} that $\Ae$ is called \emph{dense} in $\Ce$ if 
each object $X$ of $\Ce$ is a colimit of the canonical diagram $(\Ae\downarrow X)\to \Ce$.
A~category $\Ce$ is \emph{bounded} if it has a dense small full subcategory.

A category $\Ce$ is called \emph{$\lambda$\nobreakdash-accessible} if 
$\lambda$\nobreakdash-filtered colimits exist in~$\Ce$ and
there is a set $\Ae$ of $\lambda$\nobreakdash-presentable objects such that every object of $\Ce$ is a 
$\lambda$\nobreakdash-filtered colimit of objects from~$\Ae$.
A~category $\Ce$ is called \emph{accessible} if it is $\lambda$\nobreakdash-accessible for some
regular cardinal~$\lambda$.
As shown in \cite[p.\,226]{AR2} or \cite[p.\,73]{AR}, if $\Ce$ is $\lambda$\nobreakdash-accessible, 
then the full subcategory of its $\lambda$\nobreakdash-presentable objects is essentially small
and, if we denote by $\Ce_{\lambda}$ a set of representatives of all isomorphism
classes of $\lambda$\nobreakdash-presentable objects of~$\Ce$,
then $\Ce_{\lambda}$ is dense in~$\Ce$. Moreover, for every $X\in\Ce$, the slice category
$(\Ce_{\lambda}\downarrow X)$ is $\lambda$\nobreakdash-filtered and $X$ is a colimit of the
canonical diagram $(\Ce_{\lambda}\downarrow X)\to\Ce$. Thus, every accessible
category is bounded.

An accessible category is called \emph{locally presentable} if all colimits exist in it.
It then follows, by \cite[Corollary~1.28]{AR}, that all limits exist as well.
Every category of structures ${\bf Str}\,\Sigma$ is locally presentable \cite[5.1(5)]{AR},
and the forgetful functor ${\bf Str}\,\Sigma\to\Sets^S$ creates limits and colimits,
where $S$ is the set of sorts of~$\Sigma$ and $\Sets^S$ denotes the category of $S$\nobreakdash-sorted sets.

\begin{theorem}
\label{characterization}
Let $\lambda$ be a regular cardinal and let $\Ce$ be a $\lambda$\nobreakdash-accessible
category. Then there is a full embedding of $\Ce$ into a category of relational structures
that preserves $\lambda$\nobreakdash-filtered colimits.
\end{theorem}

\begin{proof}
Let us assume, with greater generality, that $\Ce$ is a bounded category
and let $\Ae$ be a dense small full subcategory of~$\Ce$.
Denote by $\Sets^{{\Ae}^{\rm op}}$ the category of functors $\Ae^{\rm op}\to\Sets$,
where $\Ae^{\rm op}$ is the opposite of~$\Ae$. Then there are full embeddings
\begin{equation}
\label{embeddings}
\Ce\longrightarrow \Sets^{\Ae^{\rm op}}\longrightarrow {\bf Str}\,\Sigma,
\end{equation}
defined as follows \cite[Ch.\,1]{AR}:
The embedding of $\Ce$ into $\Sets^{{\Ae}^{\rm op}}$ is of Yoneda type, sending
each object $X$ to the restriction of $\Ce(-,X)$ to~$\Ae^{\rm op}$.
The fact that it is full and faithful is proved in \cite[Proposition~1.26]{AR}.
The signature $\Sigma$ is chosen by picking the objects of $\Ae$ as sorts 
and the morphisms of $\Ae^{\rm op}$ as relation symbols.
The full embedding of $\Sets^{{\Ae}^{\rm op}}$ into ${\bf Str}\,\Sigma$ 
sends each functor $F$ to the $\Ae$\nobreakdash-sorted set $\{FA:A\in\Ae\}$ together with a
relation $\{(x,(Ff)x):x\in FA\}\subset FA\times FB$ for each morphism $f\colon B\to A$ in~$\Ae$.
Hence, \eqref{embeddings} sends each object $X\in\Ce$ to
\[
\left\langle \{ \Ce(A,X) : A\in\Ae \},\; 
\{ \{(\alpha,\alpha\circ f):\alpha\in\Ce(A,X)\} : f\in\Ae(B,A)\} \right\rangle.
\]

If $\Ce$ is $\lambda$\nobreakdash-accessible and we let $\Ae$ be a set
of representatives of all isomorphism classes of $\lambda$\nobreakdash-presentable objects in~$\Ce$,
then \eqref{embeddings} preserves $\lambda$\nobreakdash-filtered colimits, since 
the first arrow preserves $\lambda$\nobreakdash-filtered colimits by \cite[Proposition~1.26]{AR},
and the second arrow preserves all filtered colimits; see \cite[Example~1.41]{AR}.
\end{proof}

As in \cite[Definition~2.35]{AR}, we say that a subcategory $\Ce$ of a category $\De$ 
is \emph{accessibly embedded} if $\Ce$ is full and closed
under $\lambda$\nobreakdash-filtered colimits in $\De$ for some regular cardinal~$\lambda$.
Hence, in particular, $\Ce$ is isomorphism-closed; that is, every object of $\De$ which
is isomorphic to an object of $\Ce$ is in~$\Ce$. Moreover, the inclusion $\Ce\hookrightarrow\De$
creates $\lambda$\nobreakdash-filtered colimits.
If $\De$ is accessible and
$\Ce$ is accessibly embedded into~$\De$, then $\Ce$ is itself accessible if and only if,
for some regular cardinal~$\lambda$,
every $\lambda$\nobreakdash-filtered colimit of split subobjects of objects of $\Ce$
is in~$\Ce$; see \cite[Corollary~2.36]{AR} for details.

Vop\v{e}nka's principle implies that every full embedding between accessible
categories is accessible. The same conclusion can be inferred from the existence
of sufficiently large $\Cn$\nobreakdash-extendible cardinals~\cite{BBT}.

A~theory $T$ in a $\lambda$\nobreakdash-ary language is \emph{basic} if each of its sentences 
has the form $\forall \{x_i : i\in I\}\,(\varphi(x_i)_{i\in I}\to \psi(x_i)_{i\in I})$ 
where $\varphi$ and $\psi$ are disjunctions of positive-primitive formulas and $|I|<\lambda$.
A~formula is \emph{positive-primitive} if it has the form
$\exists \{y_j : j\in J\}\,\eta((y_j)_{j\in J},(z_k)_{k\in K})$ 
in which $\eta$ is a conjunction of atomic formulas and $|J|,|K|<\lambda$.

It follows from Theorem~\ref{characterization} that every accessible category
is equivalent to an accessibly embedded subcategory of a category of relational structures,
namely to the closure of the image of \eqref{embeddings} under isomorphisms.
Moreover, the following fundamental fact is proved in~\cite{AR}:

\begin{theorem}
\label{AdaRos}
Every accessibly embedded accessible subcategory of a category of 
structures is a category of models for some basic theory, and for every basic theory $T$ in some 
language~$\Le_{\lambda}(\Sigma)$, the category ${\bf Mod}\,T$ is accessible and accessibly
embedded into~${\bf Str}\,\Sigma$.
\end{theorem}

\begin{proof}
This is shown in \cite[Theorem~4.17 and Theorem~5.35]{AR}.
\end{proof}

We shall use the following terminology in order to simplify statements:

\begin{definition}
{\rm
An \emph{accessible category of structures} is a full subcategory
of ${\bf Str}\,\Sigma$ that is accessible and accessibly embedded,
for some signature~$\Sigma$.
}
\end{definition}

We saw in Proposition~\ref{nonabsolute} that each category ${\bf Mod}\,T$ is {\boldmath$\Delta_2$} with parameters $\{\lambda,\Sigma,T\}$. Hence, Theorem~\ref{AdaRos} implies that every accessible category of structures is at most {\boldmath$\Delta_2$}. In many cases the complexity will be lower; for example, if $\Sigma$ is finitary, then,
according to Proposition~\ref{nonabsolute}, ${\bf Mod}\,T$ is {\boldmath$\Delta_1$} with parameters $\{\Sigma,T\}$.
This amends the statement of \cite[Proposition~4.2]{BCM}.

Although, in the rest of the article, we shall restrict most of our discussion
to accessible categories of structures, results involving only concepts that are invariant 
under equivalence of categories will remain true for arbitrary accessible categories,
by Theorem~\ref{characterization}. 

A regular cardinal $\kappa$ is said to be \emph{sharply bigger} than another regular cardinal $\lambda$ if $\kappa>\lambda$ and, for each set $X$ of cardinality less than~$\kappa$, the set ${\mathcal P}_{\lambda}(X)$ has a cofinal subset of cardinality less than~$\kappa$. This notion was introduced in \cite[\S 2.3]{MP},
where it was proved that $\kappa$ is sharply bigger than $\lambda$ if and only if
every $\lambda$\nobreakdash-accessible category is $\kappa$\nobreakdash-accessible; see also \cite[Theorem~2.11]{AR}.

If $\kappa$ has the property that $\nu^{<\lambda}<\kappa$ for all $\nu<\kappa$ (which was used in Theorem~\ref{vpsigma1} above) and $\kappa>\lambda$, then $\kappa$ is sharply bigger than~$\lambda$, since, for a set $X$ of cardinality~$\nu$, the cardinality of ${\mathcal P}_{\lambda}(X)$ is precisely $\nu^{<\lambda}$. Therefore, if $\lambda\le\mu$, then
$\left(2^{\mu}\right)^+$ is sharply bigger than~$\lambda$.
This was first observed in \cite[Proposition~2.3.5]{MP} and
shows that for every $\lambda$ there are arbitrarily large regular cardinals sharply bigger than~$\lambda$.
Moreover, if $\kappa$ is strongly inaccessible and $\kappa>\lambda$, 
then $\kappa$ is sharply bigger than~$\lambda$.

In what follows, for an $S$\nobreakdash-sorted signature $\Sigma$ and a $\Sigma$\nobreakdash-structure~$A$, the \emph{cardinality} of $A$ designates the sum $\Sigma_{s\in S}\,|A_s|$ of the cardinalities of the components of its underlying $S$\nobreakdash-sorted set. 

\begin{lemma}
\label{MakkaiPare}
Let $\Sigma$ be a $\lambda$\nobreakdash-ary signature for a regular cardinal~$\lambda$, and
let $\Ce$ be a full $\lambda$\nobreakdash-accessible subcategory of ${\bf Str}\,\Sigma$ closed under $\lambda$\nobreakdash-filtered colimits.
Let $\kappa$ be a regular cardinal sharply bigger than $\lambda$ and bigger than the cardinalities of all $\lambda$\nobreakdash-presentable objects in~$\Ce$,
and such that $\Sigma\in H(\kappa)$.
Then an object $A\in\Ce$ is 
$\kappa$\nobreakdash-presentable if and only if its cardinality is smaller than~$\kappa$.
\end{lemma}

\begin{proof}
Let $S$ be the set of sorts of $\Sigma$; let $\Sigma_{\rm op}$ be its set of operation symbols and $\Sigma_{\rm rel}$ its set of relation symbols. Let $A$ be a $\Sigma$\nobreakdash-structure, and
suppose first that its cardinality $\Sigma_{s\in S}\,|A_s|$ is smaller than~$\kappa$. Let $D\colon\Ke\to\Ce$ be a $\kappa$\nobreakdash-filtered
diagram with a colimit~$L$. Then $D$ is also $\lambda$\nobreakdash-filtered and therefore the inclusion of $\Ce$ into ${\bf Str}\,\Sigma$ preserves its colimit.
Suppose given a homomorphism $f\colon A\to L$. 
Since every set $A_s$ has cardinality less than $\kappa$ and $D$ is $\kappa$\nobreakdash-filtered, each function $f_s\colon A_s\to L_s$
factors through $D(k_s)$ for some $k_s\in\Ke$. Since $|S|<\kappa$, we infer that $f$ factors (as a function) through $Dk$ for some $k\in\Ke$.
Moreover, since the cardinality of the set of all $\alpha$\nobreakdash-sequences
$\langle a_i:i\in\alpha\rangle$ with $a_i\in A_{s_i}$ for all~$i$ and with $\alpha<\lambda$ is less than~$\kappa$, and the cardinalities of the sets  $\Sigma_{\rm op}$ and $\Sigma_{\rm rel}$ are also smaller than~$\kappa$,
we can find a morphism $k\to l$ in $\Ke$ such that the composite $A\to Dk\to Dl$
is a homomorphism of $\Sigma$\nobreakdash-structures.
For the same reason, given two homomorphisms $A\to Dk$ and $A\to Dk'$ which coincide in~$L$,
there is an object $k''\in\Ke$ and morphisms $k\to k''$ and $k'\to k''$ such that the composites
$A\to Dk\to Dk''$ and $A\to Dk'\to Dk''$ are equal. Hence $A$ is $\kappa$\nobreakdash-presentable.

For the converse, by \cite[Proposition~2.3.11]{MP}, if $\kappa$ is sharply bigger than $\lambda$ then 
every $\kappa$\nobreakdash-presentable object $A$ in $\Ce$ is a 
$\lambda$\nobreakdash-filtered colimit of $\lambda$\nobreakdash-presentable objects indexed by a category with less than $\kappa$ morphisms. Therefore, since each $\lambda$\nobreakdash-presentable object has cardinality smaller than $\kappa$ and the colimit is created in~$\Sets^S$, 
it follows that $A$ also has cardinality smaller than~$\kappa$.
\end{proof}

The following is our main result in this section.

\begin{theorem}
\label{maintheorem}
Let $\Ce$ be an accessible category of structures 
and let $\Se$ be a {\boldmath$\Sigma_n$} full subcategory of~$\Ce$, where $n\ge 1$. Suppose that 
there is a proper class of supercompact cardinals if $n=2$ or that
there is a proper class of $C(n-2)$\nobreakdash-extendible cardinals
if $n\ge 3$.
Then there is a dense small full subcategory $\De\subseteq\Se$ and there are arbitrarily large regular cardinals $\kappa$
such that, for all $Y\in\Se$, the category $(\De\downarrow Y)$ is $\kappa$\nobreakdash-filtered and
$Y$ is a colimit of the canonical diagrams $(\De\downarrow Y)\to\Se$ and $(\De\downarrow Y)\to\Ce$.
\end{theorem}

\begin{proof}
Note first that, if $\Se$ is essentially small, then the result trivially holds with $\De$
a full subcategory of $\Se$ containing one representative of each isomorphism class
of objects in~$\Se$, if $\kappa$ is chosen bigger than the cardinality of the set of objects of~$\De$. 
Therefore we assume from now on that there is a proper class of nonisomorphic objects in~$\Se$.

Choose a $\Sigma_n$ formula defining $\Se$ with a set $p$ of parameters.
Suppose that $\Ce$ embeds accessibly into ${\bf Str}\,\Sigma$ for a
signature~$\Sigma$, and pick a regular cardinal $\lambda$ such that $\Sigma$ is $\lambda$\nobreakdash-ary and $\Ce$ is $\lambda$\nobreakdash-accessible and 
closed under $\lambda$\nobreakdash-filtered colimits in~${\bf Str}\,\Sigma$.
Let $\Ce_{\lambda}$ be a set of representatives of all isomorphism classes of 
$\lambda$\nobreakdash-presentable objects in~$\Ce$.

Now let $\alpha$ be any given ordinal.
Choose a regular cardinal $\kappa$ bigger than $\alpha$ and~$\lambda$, and large enough so that
 each object in $\Ce_{\lambda}$ is in~$H(\kappa)$ and $\{p,\Sigma\}\in H(\kappa)$ as well. Moreover, if $n=1$ then pick $\kappa$ of the form
$\left(2^{\mu}\right)^+$ with $\mu\ge\lambda$; if $n=2$
then choose instead $\kappa$ supercompact, and if $n\ge 3$ then choose it $C(n-2)$\nobreakdash-extendible. With any of these choices,
$\kappa$ is sharply bigger than $\lambda$ and therefore $\Ce$ is $\kappa$\nobreakdash-accessible.

Let $\De$ be a full subcategory of $\Se$ containing one representative
of each isomorphism class of objects in the set $\Se\cap H(\kappa)$.
Note that, since each object of $\De$ is
in~$H(\kappa)$, all objects of $\De$ are $\kappa$\nobreakdash-presentable in~$\Ce$,
by Lemma~\ref{MakkaiPare}.

Let $\Ce_{\kappa}$ be a set of representatives of all isomorphism classes of 
$\kappa$\nobreakdash-present\-able objects of~$\Ce$, chosen so that $\De\subseteq\Ce_{\kappa}$ and all objects of $\Ce_{\kappa}$ are in~$H(\kappa)$.
The latter is possible since, if $A\in\Ce$ and $A$ is $\kappa$\nobreakdash-presentable, then $A$ has cardinality smaller than $\kappa$ by Lemma~\ref{MakkaiPare} and therefore $A\cong A'$ as $\Sigma$\nobreakdash-structures for some $A'\in H(\kappa)$. Since $\Ce$ is isomorphism-closed, $A'$ is in $\Ce$ and we may pick $A'$ as a member of~$\Ce_{\kappa}$.

Let $Y$ be any object of~$\Se$.
Since $\Ce$ is $\kappa$\nobreakdash-accessible, we know that
$Y$ is a colimit of the canonical diagram $(\Ce_{\kappa}\downarrow Y)\to \Ce$,
which is $\kappa$\nobreakdash-filtered, by~\cite[p.\,73]{AR}.
Therefore, if we prove that $(\De\downarrow Y)$ is \emph{cofinal} 
in~$(\Ce_{\kappa}\downarrow Y)$, it will then follow that $Y$ is a colimit
of the canonical diagram $(\De\downarrow Y)\to \Ce$, and that $(\De\downarrow Y)$
is $\kappa$\nobreakdash-filtered. Moreover, since $Y$ is in~$\Se$, we shall be able to conclude 
that $Y$ is also a colimit of the canonical diagram $(\De\downarrow Y)\to \Se$, as we wanted to show.

Thus, towards proving that $(\De\downarrow Y)$ is cofinal in~$(\Ce_{\kappa}\downarrow Y)$, 
let $A$ be any object of $\Ce_{\kappa}$ and let a morphism $g\colon A\to Y$ be given. If $n=1$, then,
since $A\in H(\kappa)$, it follows from part~(a) of Theorem~\ref{vpsigma1} that
there is an object $\langle X,f\rangle$ in $(A\downarrow\Se)$ with $X\in\Se\cap H(\kappa)$,
together with an elementary embedding $e\colon X\to Y$ of $\Sigma$\nobreakdash-structures such that $e\circ f=g$. 
If $n>1$, then Theorem~\ref{theorem1} if $n=2$
or Theorem~\ref{theorem5} if $n\ge 3$ lead to the same conclusion
(recall that $H(\kappa)=V_{\kappa}$ if $\kappa$ is strongly inaccessible).
In each case, we replace, if necessary, $X$ by an isomorphic object within $\Se\cap H(\kappa)$,
so we may assume that $X\in\De$.

We therefore have a commutative triangle
\[
\xymatrix{ 
A\ar[rr]^{g}\ar[dr]_{f} &  & Y \\
& X \ar[ur]_{e} &
}
\]
where $f$ can also be viewed as 
a morphism from $\langle A,g\rangle$ to $\langle X,e\rangle$ in $(\Ce_{\kappa}\downarrow Y)$.
Since $(\Ce_{\kappa}\downarrow Y)$ is filtered, this tells us that
$(\De\downarrow Y)$ is cofinal in $(\Ce_{\kappa}\downarrow Y)$, as we wanted to show.
\end{proof}

\begin{corollary}
If there is a proper class of supercompact cardinals, then every accessible
category is co-wellpowered.
\end{corollary}

\begin{proof}
Let $\Ce$ be an accessible category. Since accessibility and co-well\-power\-edness are invariant under equivalence of categories,
we can assume that $\Ce$ is a category of models of a basic theory $T$ for some signature~$\Sigma$, by Theorem~\ref{characterization} and Theorem~\ref{AdaRos}.

For an object $A\in\Ce$, let
$\Ee_A$ be the full subcategory of $(A\downarrow \Ce)$ whose objects are the
epimorphisms. Then $\Ee_A$ is a partially ordered class, since between any two of its objects there is at most one morphism. Moreover, $\Ee_A$ is closed under colimits in $(A\downarrow \Ce)$ and, if a diagram $D\colon \Ke\to\Ee_A$ has a colimit, then the colimit is a supremum of the set $\{Dk:k\in\Ke\}$, hence determined by this set up to isomorphism. Therefore, in order to prove that $\Ce$ is co\nobreakdash-well\-powered, it is enough to prove that $\Ee_A$ is bounded for every~$A$, since this implies that $\Ee_A$ is essentially small.

From the fact that $\Ce$ is {\boldmath{$\Delta_2$}} it follows that $\Ee_A$ is {\boldmath{$\Pi_2$}}, since an object of $\Ee_A$ is a pair $\langle Y,g\rangle$ where $g\in\Ce(A,Y)$ and
\[
\begin{array}{c}
\forall Z\,\forall h\,\forall h'\, 
[(h\in\Ce(Y,Z)\,\wedge\, h'\in\Ce(Y,Z)\,\wedge\, h\circ g=h'\circ g)\,\to\, h=h'],
\end{array}
\]
and a morphism $\langle Y,g\rangle \to \langle Y',g'\rangle$ is a morphism $d\in\Ce(Y,Y')$ with $g'=d\circ g$.
Hence, Theorem~\ref{maintheorem} implies that $\Ee_A$ is bounded under the assumption that there are arbitrarily large \emph{extendible} cardinals.

However, as we next show, it is enough to assume that there are arbitrarily large \emph{supercompact} cardinals. For this, we need to repeat the argument used in the proof of Theorem~\ref{maintheorem} and the one used in the proof of Theorem~\ref{theorem1}, adapted to our current situation.

If $\Ce$ is accessible, then $(A\downarrow\Ce)$ is also accessible, by \cite[Corollary~2.44]{AR}. Pick a regular cardinal $\lambda$ such that $(A\downarrow\Ce)$ is $\lambda$\nobreakdash-accessible.
Assuming that there exists a proper class of supercompact cardinals, we may choose a supercompact cardinal $\kappa$ bigger than $\lambda$, such that $\Sigma,T\in H(\kappa)$ and such that all $\lambda$\nobreakdash-presentable objects of $(A\downarrow\Ce)$ are in $H(\kappa)$. Since $\kappa$ is strongly inaccessible, it is sharply bigger than $\lambda$ and therefore $(A\downarrow\Ce)$ is $\kappa$\nobreakdash-accessible.

Choose a full subcategory $\De$ of $\Ee_A$ containing one representative of each isomorphism class of objects in $\Ee_A\cap H(\kappa)$. By Lemma~\ref{MakkaiPare}, all objects in $\De$ are $\kappa$\nobreakdash-presentable. Choose also a set $(A\downarrow\Ce)_{\kappa}$ of representatives of all isomorphism classes of $\kappa$\nobreakdash-presentable objects of $(A\downarrow\Ce)$, containing $\De$ and such that all its objects are in $H(\kappa)$, which is possible by Lemma~\ref{MakkaiPare}.

Now let $\langle Y,g\rangle$ be any object of $\Ee_A$, so $g\colon A\to Y$ is an epimorphism. We know that $\langle Y,g\rangle$ is a colimit of the canonical diagram
\[
((A\downarrow\Ce)_{\kappa}\downarrow \langle Y,g\rangle)\longrightarrow (A\downarrow\Ce).
\]
Hence it suffices to prove that $(\De\downarrow \langle Y,g\rangle)$ is cofinal in 
$((A\downarrow\Ce)_{\kappa}\downarrow \langle Y,g\rangle)$. For this, pick any object in $((A\downarrow\Ce)_{\kappa}\downarrow \langle Y,g\rangle)$, which consists of a $\kappa$\nobreakdash-pres\-ent\-able object $\langle B,a\rangle$ of $(A\downarrow\Ce)$ together with a morphism $d\colon B\to Y$ such that $d\circ a=g$. Pick a cardinal $\mu>\kappa$ such that $\langle Y,g\rangle\in H(\mu)$. Then $d$ is also in~$H(\mu)$ since $B\in H(\kappa)$.

Let $j\colon V\to M$ be an elementary embedding with $M$ transitive and critical point~$\kappa$, such that $j(\kappa)>\mu$ and $M$ is closed under $\mu$\nobreakdash-sequences. Then $g$ and $d$ are in $M$ since $H(\mu)\in M$.
Moreover, $\Ce$ is absolute between $M$ and~$V$, by part~(b) of Proposition~\ref{nonabsolute}. Therefore $g$ is also an epimorphism in~$M$,
since, if $h,h'\in\Ce(Y,Z)$ satisfy $h\circ g=h'\circ g$ in~$M$, then $h$ and $h'$ also belong to $\Ce(Y,Z)$ in $V$ and therefore $h=h'$, since $g$ is an epimorphism in~$V$.

Since $Y\in H(\mu)$, the restriction $j\restriction Y:Y\to j(Y)$ is in~$M$, and it is an elementary embedding of $\Sigma$\nobreakdash-structures by Theorem~\ref{lemma0}.
Since $A$ and $B$ are in $H(\kappa)$, we have $j(A)=A$ and $j(B)=B$. Therefore, as in the proof of Theorem~\ref{theorem1}, $g\colon A\to Y$ and $d\colon B\to Y$ witness that in $M$ there exists an object $X$ (namely, $Y$) and an epimorphism $f\in\Ce(A,X)$ with $\text{rank}(X)<j(\kappa)$, together with an elementary embedding $e\colon X\to j(Y)$ such that $e\circ f=j(g)$ and a morphism $c\in\Ce(B,X)$ such that $c\circ a=f$ and $e\circ c=j(d)$. This implies, by elementarity of~$j$, that in $V$ there is an epimorphism $f\in\Ce(A,X)$ with $\text{rank}(X)<\kappa$, together with an elementary embedding $e\colon X\to Y$ such that $e\circ f=g$ and a morphism $c\in\Ce(B,X)$ such that $c\circ a=f$ and $e\circ c=d$. In other words, there is a commutative diagram
\[
\xymatrix{ 
& A \ar[dl]_a \ar[d]^f \ar[dr]^g & \\
B \ar@/_1.5pc/[rr]_d \ar[r]_c & X \ar[r]_e  & Y.
}
\]

Here we may replace $\langle X,f\rangle$ by an isomorphic object which is in $\De$.
This shows that $(\De\downarrow \langle Y,g\rangle)$ is cofinal in 
$((A\downarrow\Ce)_{\kappa}\downarrow \langle Y,g\rangle)$, and consequently the category $\Ee_A$ is bounded, as needed.
\end{proof}

On the other hand, as shown in~\cite[A.19]{AR}, if each accessible category is co-wellpowered then 
there exists a proper class of measurable cardinals. Therefore, the statement that 
every accessible category is co-wellpowered is set-theoretical. Its precise consistency strength is not known; 
see~\cite[Open Problem~11]{AR}. By part~(i) of \cite[Theorem~6.3.8]{MP}, together with the fact that
categories of epimorphisms can be sketched by a pushout sketch (as done in~\cite[p.\,101]{AR}), the statement
that every accessible category is co-wellpowered is implied by
the existence of a proper class of strongly compact cardinals,
a large-cardinal assumption that is not known to be 
weaker, consistency-wise, than the existence of a proper class of supercompact cardinals. 

In order to simplify the statements of several corollaries of Theorem~\ref{maintheorem},
we shall use the following terminology. 

\begin{definition}
{\rm
We say that a class $\Se$ is \emph{definable with sufficiently low complexity}
if any of the following conditions is satisfied:
\begin{itemize}
\item[{\rm (1)}]
$\Se$ is {\boldmath$\Sigma_1$}.
\item[{\rm (2)}]
There is a proper class of supercompact cardinals and $\Se$ is~{\boldmath$\Sigma_2$}.
\item[{\rm (3)}]
There is a proper class of $\Cn$\nobreakdash-extendible cardinals for some $n\ge 1$ and $\Se$ 
is~{\boldmath$\Sigma_{n+2}$}.
\end{itemize}
}
\end{definition}

By Corollary~\ref{corollary2bis}, 
if Vop\v{e}nka's principle holds, then all classes
are definable with sufficiently low complexity.

\section{Small-orthogonality classes}
\label{smallorthogonalityclasses}

An object $X$ and a morphism $f\colon A\to B$ in a category $\Ce$
are called \emph{orthogonal} \cite{FK} if the function
\[
\Ce(f,X)\colon\Ce(B,X)\longrightarrow\Ce(A,X)
\]
is bijective. That is, $X$ and $f$ are orthogonal if and only if for every
morphism $g\colon A\to X$ there is a unique morphism $h\colon B\to X$ such that
$h\circ f=g$. 

For a class of objects~$\Xe$, we denote by
${}^{\perp}\Xe$ the class of morphisms that are orthogonal
to all the objects of~$\Xe$. Similarly, for a class of 
morphisms~$\Fe$, we denote by $\Fe^{\perp}$ the class of objects
that are orthogonal to all the morphisms of~$\Fe$.
Classes of objects of the form $\Fe^{\perp}$ are called
\emph{orthogonality classes}, and, if $\Fe$ is a set
(not a proper class), then $\Fe^{\perp}$ is a \emph{small-orthogonality class}.

In what follows, we view each class of morphisms in $\Ce$
as a full subcategory of the category of arrows~$\Ar\Ce$. 

\begin{lemma}
\label{dense}
For a regular cardinal~$\lambda$,
let $\Fe$ be a class of morphisms in a $\lambda$\nobreakdash-accessible 
category~$\Ce$, and let $\De\subseteq\Fe$.
Suppose that every $f\in\Fe$ is a $\lambda$\nobreakdash-filtered colimit
of elements of~$\De$, and suppose that the inclusion of $\Fe$ into $\Ar\Ce$ preserves the colimit.
Then $\De^{\perp}=\Fe^{\perp}$.
\end{lemma}

\begin{proof}
To prove this claim, only the inclusion
$\De^{\perp}\subseteq \Fe^{\perp}$ needs to be checked.
Let $X\in\De^{\perp}$ and let $f\colon A\to B$ be any element of~$\Fe$.
By assumption, $f=\colim\,d_k$ where $d_k\colon A_k\to B_k$
is in $\De$ for all~$k\in\Ke$, and $\Ke$ is $\lambda$\nobreakdash-filtered.
Since $\Ce$ is $\lambda$\nobreakdash-accessible, the colimits
$\colim\,A_k$ and $\colim\,B_k$ exist, and the induced
arrow $g\colon\colim\,A_k\to\colim\,B_k$ is a colimit
of the arrows $d_k$ in~$\Ar\Ce$. Since $f$ is also a colimit of
the same diagram, we infer that $g\cong f$. Hence, $f$ induces bijections
\begin{align*}
\Ce(B,X) & \cong \Ce(\colim\,B_k, X)\cong \lim\,\Ce(B_k,X) \\
& \cong \lim\,\Ce(A_k,X) \cong \Ce(\colim\,A_k,X) \cong \Ce(A,X),
\end{align*}
which means that $X\in\Fe^{\perp}$, as needed.
\end{proof}

\begin{lemma}
\label{ortocount}
If $\Se$ is a {\boldmath$\Sigma_{n+1}$} full subcategory of a 
{\boldmath$\Sigma_{n}$} category~$\Ce$, then
${}^{\perp}\Se$ is {\boldmath$\Pi_{n+1}$} if $n\ge 1$,
and it is {\boldmath$\Pi_2$} if $n=0$.
\end{lemma}

\begin{proof}
The class of morphisms ${}^{\perp}\Se$ can be defined as follows: 
$\langle A,B,f\rangle\in {}^{\perp}\Se$ if and only if
\begin{equation}
\label{orthogonality}
\begin{array}{c}
\forall X \, \forall g \, [(X\in\Se\, \wedge\, g\in\Ce(A,X))
\to \exists h\,(h\in\Ce(B,X)\,\wedge\, h\circ f = g)] \\[0.1cm]
\wedge \, \forall X\,\forall h_1\,\forall h_2\, [(X\in\Se\,\wedge\, h_1\in\Ce(B,X)\,\wedge\,h_2\in\Ce(B,X) \\[0.1cm] 
\wedge\,h_1\circ f=h_2\circ f)\to h_1=h_2].
\end{array}
\end{equation}
Recall that $P\to Q$ means $\neg(P\wedge\neg Q)$, or $\neg P\vee Q$.
Therefore, \eqref{orthogonality} is at least~$\Pi_2$,
and it is $\Pi_{n+1}$ if $\Se$ is~{\boldmath$\Sigma_{n+1}$} 
and $\Ce$ is at most {\boldmath$\Sigma_{n}$} with $n\ge 1$.
\end{proof}

\begin{theorem}
\label{theorem6.3}
Assume the existence of a proper class of $\Cn$\nobreakdash-extendible cardinals,
where $n\ge 2$. Then each {\boldmath$\Sigma_{n+1}$}
orthogonality class in an accessible category $\Ce$ 
of structures is a small-orthogonality class.
\end{theorem}

\begin{proof}
Let $\Se$ be a full subcategory of $\Ce$ whose objects form a 
{\boldmath$\Sigma_{n+1}$} orthogonality class. Thus $\Se=\Fe^{\perp}$ for some~$\Fe$, and
this implies that 
\[
({}^{\perp}\Se)^{\perp}=({}^{\perp}(\Fe^{\perp}))^{\perp}=\Fe^{\perp}=\Se.
\]

Since $\Ce$ is {\boldmath$\Delta_2$} by Proposition~\ref{nonabsolute}, we infer from
Lemma~\ref{ortocount} that ${}^{\perp}\Se$ is {\boldmath$\Pi_{n+1}$}.
Now the category of arrows $\Ar\Ce$ is accessible and embeds accessibly into a category
of structures in such a way that complexity is preserved, by Lemma~\ref{relativizing}.
Hence, by Theorem~\ref{maintheorem}, ${}^{\perp}\Se$ has a
dense small full subcategory $\De$ and there is a regular cardinal
$\kappa$ (which we may choose so that $\Ce$ is $\kappa$\nobreakdash-accessible)
such that every arrow $f\in {}^{\perp}\Se$ is a $\kappa$\nobreakdash-filtered
colimit of elements of~$\De$, both in ${}^{\perp}\Se$ and in~$\Ar\Ce$.
Then $\De^{\perp}=({}^{\perp}\Se)^{\perp}=\Se$
by Lemma~\ref{dense}, so $\Se$ is indeed a small-orthogonality class.
\end{proof}

This result can be sharpened as follows.
A~\emph{reflection} on a category is a left adjoint (when it exists) of the inclusion
of a full subcategory~\cite{Mac}, which is then called \emph{reflective}. 
For example, in the category of groups, 
the abelianization functor is a reflection onto the reflective full subcategory of abelian groups.
For every reflection~$L$, the closure under isomorphisms of its image
is an orthogonality class, and it is in fact orthogonal to the 
class of \emph{$L$\nobreakdash-equivalences}, i.e., morphisms $f$ such that $Lf$ is an isomorphism.

A~reflection $L$
is called an \emph{$\Fe$\nobreakdash-reflection}, where $\Fe$ is a set or a proper class of morphisms,
if the closure under isomorphisms of the image of $L$ is equal to~$\Fe^{\perp}$. This notion is
particularly relevant when $\Fe$ can be chosen to be a set (or even better a single morphism).
In the previous example, abelianization is an $f$\nobreakdash-reflection where $f$ is 
the canonical projection of a free group on two generators onto
a free abelian group on two generators, since the groups orthogonal to $f$ are precisely the abelian groups.

\begin{theorem}
\label{corollary3.4}
Let $L$ be a reflection on an accessible category $\Ce$ of structures.
Then $L$ is an $\Fe$\nobreakdash-reflection for some set $\Fe$ of morphisms
under any of the following assumptions:
\begin{itemize}
\item[(1)]
The class of $L$\nobreakdash-equivalences is definable with sufficiently low complexity.
\item[(2)]
The class of objects isomorphic to $LX$ for some $X$ is {\boldmath$\Sigma_{n+1}$}
for $n\ge 2$ and there is a proper class of $\Cn$\nobreakdash-extendible cardinals.
\end{itemize}
\end{theorem}

\begin{proof}
To prove case~(1), let $\Se$ be the full subcategory of $L$\nobreakdash-equivalences
in the category of arrows of~$\Ce$.
It then follows from Theorem~\ref{maintheorem}
that there is a small full subcategory $\De$ of $\Se$ which is dense
and satisfies $\Se^{\perp}=\De^{\perp}$, by Lemma~\ref{dense}, as needed.
Case~(2) follows as a special case of Theorem~\ref{theorem6.3}.
\end{proof}

The following corollary is a stronger variant of
\cite[Corollary~4.6]{BCM}. The assumptions that $L$ be an epireflection and
that $\Ce$ be balanced, which were made in~\cite{BCM}, are not at all necessary here.

\begin{corollary}
\label{onemore}
Suppose that there is a proper class of supercompact cardinals.
If $L$ is a reflection on an accessible category $\Ce$ of structures
and the class of $L$\nobreakdash-equivalences is~{\boldmath$\Sigma_2$},
then $L$ is an $\Fe$\nobreakdash-reflection for some set $\Fe$ of morphisms.
\end{corollary}

\begin{proof}
By assumption, the class of $L$\nobreakdash-equivalences is definable with 
sufficiently low complexity. Hence, Theorem~\ref{corollary3.4} applies.
\end{proof}

As already shown in \cite[Theorem~6.3]{CSS}, the assertion that every reflection on an
accessible category is an $\Fe$\nobreakdash-reflection for some set $\Fe$ of morphisms
cannot be proved in~ZFC. Specifically,
if one assumes that measurable cardinals do not exist and considers reflection
on the category of groups with respect to the class $\Ze$ of homomorphisms
of the form $\Z^{\kappa}/\Z^{<\kappa}\to\{0\}$, where $\kappa$ runs
over all cardinals (see Example~\ref{example2}), then there is no set $\Fe$ of group homomorphisms
such that $\Fe$\nobreakdash-reflection coincides with $\Ze$\nobreakdash-reflection. 
This fact was also used in~\cite{BCM}.

\begin{theorem}
\label{corollary6.4}
If $\Ce$ is a locally presentable category  
of structures, then every full subcategory $\Se$ of $\Ce$ closed under limits and
definable with sufficiently low complexity is reflective.
\end{theorem}

\begin{proof}
As in the proof of Theorem~\ref{maintheorem}, for every $A\in\Ce$
we can choose a small full subcategory $\De$ of $\Se$ (depending on the cardinality of $A$ and the parameters of $\Ce$)
such that every arrow $f\colon A \to Y$ 
with $Y$ in $\Se$ factors through some object $X\in\De$.
Hence the inclusion functor $\Se\hookrightarrow\Ce$ satisfies the
solution-set condition for every $A$ in~$\Ce$, 
as required in the Freyd Adjoint Functor Theorem \cite[V.6]{Mac},
from which the existence of a reflection of $\Ce$ onto $\Se$ follows.
\end{proof}

The following result is a further improvement, since it implies, among other things, 
that, if $\Se$ is~{\boldmath$\Sigma_1$},
then the reflectivity of $\Se^{\perp}$ is provable in ZFC.
This yields, in particular, a solution of the Freyd--Kelly orthogonal subcategory problem
\cite{FK} in ZFC for {\boldmath$\Sigma_1$} classes.

\begin{theorem}
\label{corollary6.5}
Let $\Se$ be a class of morphisms definable with sufficiently low complexity in an accessible category $\Ce$ 
of structures. Then $\Se^{\perp}$ is a small-orthogonality class and, if $\Ce$ is cocomplete,
then $\Se^{\perp}$ is reflective.
\end{theorem}

\begin{proof}
If we view $\Se$ as a full subcategory of the category of arrows of~$\Ce$, then
Theorem~\ref{maintheorem} ensures that $\Se$ has a dense small full subcategory $\De$
and Lemma~\ref{dense} implies that $\De^{\perp}=\Se^{\perp}$. Hence $\Se^{\perp}$
is a small-orthogonality class, and small-orthogonality classes are reflective 
if colimits exist \cite[1.37]{AR}.
\end{proof}

If we weaken the assumption that $\Se$ is closed under limits in 
Theorem~\ref{corollary6.4}, by imposing only that it is closed under products and retracts,
then we may infer similarly that $\Se$ is weakly reflective, under the hypotheses made in the statement.
On the other hand, it is shown in \cite{CGR} that, assuming the nonexistence of measurable cardinals, there is 
a {\boldmath$\Sigma_2$} full subcategory $\Se$ of the category of abelian groups which is closed under 
products and retracts but not weakly reflective. Specifically, $\Se$ is the closure
of the class of groups $\Z^\kappa/\Z^{<\kappa}$ under products and retracts, where $\kappa$ runs over all cardinals.
Hence, the statement that all
{\boldmath$\Sigma_2$} full subcategories closed under products and retracts
in locally presentable categories are weakly reflective
implies the existence of measurable cardinals, while it follows from the existence
of supercompact cardinals.

\begin{theorem} 
Every full subcategory closed under colimits and definable with sufficiently low complexity
in a locally presentable category $\Ce$ of structures is coreflective.
\end{theorem}

\begin{proof} 
Argue as in \cite[Theorem~6.28]{AR}. 
\end{proof}

\section{Consequences in homotopy theory}
\label{cohomologicallocalizations}

Hovey conjectured in \cite{Hov} that for every cohomology theory
defined on spectra there is a homology theory with the same acyclics. 
This conjecture remains so far unsolved.
In a different but closely related direction, the existence of cohomological
localizations is also an open problem in ZFC, although it is known
that it follows from Vop\v{e}nka's principle, both in unstable homotopy
and in stable homotopy, by \cite{CSS} and~\cite[Theorem~1.5]{CCh}.

Motivated by these problems,
in this section we compare homological acyclic classes with
cohomological acyclic classes from the point of view of complexity
of their definitions.
We consider homology theories and cohomology theories defined on
simplicial sets and represented by spectra.

Spectra will be meant in the sense of Bousfield--Friedlander \cite{BF}.
Thus, a \emph{spectrum} $E$ is a sequence of pointed simplicial sets
\[
\langle(E_n,p_n) : \,p_n\in(E_n)_0,\;0\le n<\omega\rangle
\]
equipped with pointed simplicial maps $\sigma_n\colon S E_{n}\to E_{n+1}$ for all~$n$.
Here $S$ denotes \emph{suspension}, that is, $S X=\Sph^1\wedge X$. For $k\ge 1$,
we denote by $\Sph^{k}$ the simplicial $k$\nobreakdash-sphere, namely
$\Sph^{k}=\Delta[k]/\partial\Delta[k]$, where $\Delta[k]$ is the standard $k$\nobreakdash-simplex
and $\partial\Delta[k]$ is its boundary.
For pointed simplicial sets $X$ and~$Y$,
the \emph{smash product} $X\wedge Y$ is the quotient of the product $X\times Y$ by the
wedge sum $X\vee Y$, and we
denote by $\Map_*(X,Y)$ the \emph{pointed function complex} from $X$ to~$Y$, whose $n$\nobreakdash-simplices are
the pointed maps $X\wedge\Delta[n]_+\to Y$, where the subscript $+$ means that a disjoint basepoint has been added.

A simplicial set is \emph{fibrant} if it is a Kan complex \cite{Kan}. 
For the purposes of this article, it will be convenient to use
Kan's ${\rm Ex}^{\infty}$ construction
as a fibrant replacement functor. Thus, there is a natural (injective)
weak equivalence $j_Y\colon Y\hookrightarrow \Exinf Y$ for all~$Y$,
where $\Exinf Y$ is fibrant.

Let $[X,Y]$ denote the set of morphisms from $X$ to $Y$ in the pointed
homotopy category of simplicial sets, which
can be described as the set of pointed homotopy classes of maps $X\to\Exinf Y$.
If $Y$ is fibrant, then this is in bijective correspondence, via~$j_Y$, with 
the set of pointed homotopy classes of maps $X\to Y$.

A~spectrum $E$ is an \emph{$\Omega$\nobreakdash-spectrum} if each $E_n$ is fibrant and the
adjoints $\tau_n\colon E_n\to \Omega E_{n+1}$ of the structure maps $\sigma_n\colon SE_n\to E_{n+1}$ are weak equivalences,
where $\Omega$ denotes the \emph{loop space} functor $\Omega X=\Map_*(\Sph^1,X)$.

Each spectrum $E$ defines a reduced homology theory $E_*$ on simplicial sets~by
\begin{equation}
\label{homologies}
E_k(X)=\colim_n\,\pi_{n+k}(X\wedge E_n) = \colim_n\,[\Sph^{n+k}, X\wedge E_n]
\end{equation}
for $k\in\Z$,
and, if $E$ is an $\Omega$\nobreakdash-spectrum, then $E$ defines a reduced cohomology theory $E^*$
on simplicial sets by
\begin{equation}
\label{cohomologies}
E^k(X)=\colim_n\,\pi_{n-k}(\Map_*(X,E_n)) = \colim_n\,[S^nX,E_{n+k}]
\end{equation}
for $k\in\Z$. Note that, if $k\ge 0$, then simply $E^k(X)\cong [X,E_k]$.

Such homology or cohomology theories are called \emph{representable},
and we shall only consider these in this article. Although not
every generalized homology or cohomology theory in the sense of Eilenberg--Steenrod is representable
\cite[Example~II.3.17]{Rud}, homological localizations have only been constructed and studied
assuming representability \cite{AF}, \cite{BouTop}. According to Brown's representability theorem,
every cohomology theory which is \emph{additive} (i.e., sending coproducts to products)
is represented by some $\Omega$\nobreakdash-spectrum. Similarly, homology theories
that preserve filtered colimits are representable. 
See \cite{Ada} or \cite{Rud} for further details.

In most of what follows, we assume that $E$ is an $\Omega$\nobreakdash-spectrum.
A~simplicial set $X$ is called \emph{$E_*$\nobreakdash-acyclic}
if $E_k(X)=0$ for all $k\in\Z$, and, similarly, 
$X$ is \emph{$E^*$\nobreakdash-acyclic} if $E^k(X)=0$ for all $k\in\Z$.
Observe that, by~\eqref{cohomologies}, the statement that $X$ is $E^*$\nobreakdash-acyclic is equivalent
to the statement that the pointed function complex $\Map_*(X,E_n)$
is weakly contractible (that is, connected and with vanishing homotopy groups) for all~$n$.

A~map $f\colon X\to Y$ is an \emph{$E_*$\nobreakdash-equivalence} if 
\[
E_k(f)\colon E_k(X)\longrightarrow E_k(Y)
\] 
is an isomorphism of abelian groups for all $k\in\Z$,
and similarly for cohomology.
Let $Cf$ denote the \emph{mapping cone} of~$f$, which 
is obtained from the disjoint union of $Y$ and $X\times \Delta[1]$ by identifying
$X\times\{0\}$ with $f(X)\subseteq Y$ using~$f$, and collapsing
$X\times\{1\}$ to a point. Using the Mayer--Vietoris axiom, one finds that
$f$ is an $E_*$\nobreakdash-equivalence if and only if $Cf$ is $E_*$\nobreakdash-acyclic, 
and analogously for cohomology.

The category of simplicial sets is~$\Delta_0$, locally presentable, and it has a canonical accessible embedding
into a category of structures with a finitary $\omega$\nobreakdash-sorted operational signature.
In fact, one can write down explicitly a formula without unbounded quantifiers 
expressing that $X$ and $Y$ are simplicial sets and $f$ is a simplicial map from $X$ to~$Y$.
This amounts to formalizing the claim that a simplicial set $X$
is a sequence of sets $\langle X_n:0\le n<\omega\rangle$ (where the elements
of $X_n$ are called \emph{$n$\nobreakdash-simplices}), 
together with functions $d_i^n\colon X_n\to X_{n-1}$ (called \emph{faces}) 
for $n\ge 1$ and $0\le i\le n$, and
$s_i^n\colon X_n\to X_{n+1}$ (called \emph{degeneracies})
for $n\ge 0$ and $0\le i\le n$, satisfying the simplicial identities;
see \cite[Definition~1.1]{May}. A~simplicial map $f\colon X\to Y$
is a sequence of functions $\langle f_n\colon X_n\to Y_n\rangle_{0\le n<\omega}$ compatible 
with faces and degeneracies.

Similarly, the category of spectra is~$\Delta_0$, locally presentable, and it also has an accessible embedding
into a category of structures with a finitary $\omega$\nobreakdash-sorted operational signature,
since a spectrum $E$ consists of a sequence of pointed simplicial sets $\langle (E_m,p_m):0\le m<\omega\rangle$,
where $p_m\in (E_m)_0$, and a sequence of pointed maps $\langle \sigma_m\colon SE_m\to E_{m+1}\rangle_{0\le m<\omega}$, each of which can be viewed as a map $\Delta[1]\times E_m\to E_{m+1}$ sending
$\partial\Delta[1]\times E_m$ and $\Delta[1]\times \{p_m\}$ to the basepoint~$p_{m+1}$.
Giving a map $f\colon \Delta[1]\times E_m\to E_{m+1}$
is equivalent to giving a collection of functions 
\[
f_0^0,f_0^1\colon(E_m)_0\to(E_{m+1})_0 \quad \mbox{and} \quad
f_k^0,f_k^1,f_k^{01}\colon(E_m)_k\to(E_{m+1})_k
\]
for~$k\ge 1$, with commutativity conditions
\[
\begin{array}{lll}
f_0^0\circ d_0^1=d_0^1\circ f_1^0, & \quad
f_0^1\circ d_0^1=d_0^1\circ f_1^1, & \quad
f_0^0\circ d_0^1=d_0^1\circ f_1^{01}, \\[0.2cm]
f_0^0\circ d_1^1=d_1^1\circ f_1^0, & \quad
f_0^1\circ d_1^1=d_1^1\circ f_1^1, & \quad
f_0^1\circ d_1^1=d_1^1\circ f_1^{01}, \\[0.2cm]
s_0^0\circ f_0^0=f_1^0\circ s_0^0, & \quad
s_0^0\circ f_0^1=f_1^1\circ s_0^0, 
\end{array}
\]
and correspondingly for $k\ge 1$.

\begin{proposition}
\label{weqs}
The following are $\Delta_1$ classes:
\begin{itemize}
\item[{\rm (1)}] Fibrant simplicial sets.
\item[{\rm (2)}] Weak equivalences of simplicial sets.
\item[{\rm (3)}] Weakly contractible spectra.
\item[{\rm (4)}] $\Omega$\nobreakdash-spectra.
\end{itemize}
\end{proposition}

\begin{proof}
The assertion that a given simplicial set $X$ is fibrant can be formalized
by means of the Kan extension condition, as in~\cite[Definition~1.3]{May}.
Explicitly, a simplicial set $X$ is fibrant if and only if
for every $1\le n<\omega$ and every $k\leq n+1$, the following sentence holds: 
For all $x_0,x_1,\dots,x_{n+1}\in X_n$ such that $d_i^n x_j=d_{j-1}^n x_i$ for $i<j$, $i\ne k$ 
and $j\ne k $, there exists $x\in X_{n+1}$ such that $d_i^{n+1}x=x_i$ for $i\ne k$. 
Since quantification over finite subsets is~$\Delta_1$ (see Example~\ref{example3}), 
the class of fibrant simplicial sets is $\Delta_1$\nobreakdash-definable.

Towards~(2), recall that a map of simplicial sets $f\colon X\to Y$ is a weak equivalence if and only if
it induces a bijection of connected components and isomorphisms of homotopy groups
for every choice of a basepoint.
Let us assume first that $X$ and $Y$ are fibrant. 
Then $f$ induces a bijection of connected components
if and only if, for all $x_0$ and $x_1$ of $X_0$, if there exists $v\in Y_1$ with $d_0^1v=f(x_0)$
and $d_1^1v=f(x_1)$, then there exists $u\in X_1$ with $d_0^1u=x_0$ and $d_1^1u=x_1$,
and moreover for each $y\in Y_0$ there exist $x\in X_0$ and $v\in Y_1$ such that
$d_0^1v=y$ and $d_1^1v=f(x)$.
Hence, the statement that $f$ induces a bijection of
connected components is~$\Delta_0$.

Similarly, if a simplicial set $X$ is fibrant, then the $n$th homotopy group
$\pi_n(X,p)$ with basepoint $p\in X_0$ is the quotient of the set of all $x\in X_n$ such that
$d_i^nx=sp$ for all~$i$ (where $s=s_{n-2}^{n-2}\circ\cdots\circ s_0^0$) by the homotopy relation, where
$x\sim x'$ if $d_i^nx=d_i^nx'$ for all $i$ and there exists $z\in X_{n+1}$ with
$d_{n+1}^{n+1}z=x$, $d_n^{n+1}z=x'$, and $d_i^{n+1}z=s_{n-1}d_i^nx$ for $0\le i<n$; compare with \cite[Definition~3.1]{May}.
Therefore, if $X$ and $Y$ are fibrant, then $f$ induces an isomorphism
$\pi_n(X,p)\cong\pi_n(Y,q)$, where $p\in X_0$ and $q=f(p)$, if and only if the
following sentence holds:
\[
\begin{array}{c}
\forall y\in Y_n \, [\forall i\le n\,(d_i^ny=sq) \to [\exists x\in X_n \,(\forall i\le n \,
(d_i^nx=sp)
\\[0.2cm]
\wedge \, f_n(x)\sim y \,
\wedge \, \forall x'\in X_n \, ((\forall i\le n\, (d_i^nx'=sp)\, \wedge \, f_n(x')\sim y) \to x\sim x'))]].
\end{array}
\]
This shows that the statement that a map between \emph{fibrant} simplicial sets
is a weak equivalence is~$\Delta_1$.

Next we analyze the complexity of a fibrant replacement.
For a simplicial set~$X$, the map $j_X\colon X\hookrightarrow\Exinf X$ can be defined
as the inclusion of $X$ into a simplicial set $\Exinf X$ defined as follows.
Let ${\rm Ex}^1 X$ be the simplicial set whose set
of $n$\nobreakdash-simplices is the set of all maps from the barycentric subdivision of $\Delta[n]$
into~$X$. The barycentric subdivision ${\rm sd}\,\Delta[n]$ is the nerve of the poset of nondegenerate
simplices of $\Delta[n]$ (see \cite[Ch.\,III, \S 4]{GJ}).
The \emph{last vertex map} ${\rm sd}\,\Delta[n]\to\Delta[n]$
yields an inclusion $X\hookrightarrow {\rm Ex}^1 X$.
Then $\Exinf X$ is the union of a sequence of inclusions ${\rm Ex}^k X\hookrightarrow {\rm Ex}^{k+1} X$
for $k\ge 1$, where ${\rm Ex}^k$ is the composite of ${\rm Ex}^1$ with itself $k$~times.

Let $p$ be any vertex of~$X$. Each element in $\pi_n(\Exinf Y,f(p))$ is represented by a map $\Sph^n\to{\rm Ex}^k Y$
based at $f(p)$ for some $k<\omega$, that is, a map from $\Delta[n]$ to ${\rm Ex}^k Y$ sending the boundary of 
$\Delta[n]$ to~$f(p)$.
By adjointness, the maps $\Delta[n]\to {\rm Ex}^k Y$ correspond bijectively with the
maps ${\rm sd}^k\Delta[n]\to Y$, where ${\rm sd}^k$ is an iterated barycentric subdivision.
Let $a_{k,n}$ be the number of nondegenerate $n$\nobreakdash-simplices of ${\rm sd}^k\Delta[n]$
and let $R_{k,n}$ be the set of all relations among their faces. 
For example, $a_{2,1}=4$ and $R_{2,1}$ consists of the equalities 
\[
\begin{array}{c}
d_1^1\,x_{\mbox{\tiny{$(0\to 001)$}}}  = d_1^1\,x_{\mbox{\tiny{$(01\to 001)$}}}, \qquad
d_0^1\,x_{\mbox{\tiny{$(01\to 001)$}}} = d_0^1\,x_{\mbox{\tiny{$(01\to 011)$}}},\\[0.2cm]
d_1^1\,x_{\mbox{\tiny{$(01\to 011)$}}} = d_1^1\,x_{\mbox{\tiny{$(1\to 011)$}}}.
\end{array}
\]
Thus, each map $\Delta[n]\to {\rm Ex}^kY$ is determined by a sequence of
$a_{k,n}$ (not necessarily distinct) elements of $Y_n$
satisfying a set $R_{k,n}$ of equalities among their faces. 
In what follows, when we write ``a map $\beta\colon\Sph^n\to{\rm Ex}^kY$''
we implicitly formalize it as an ordered sequence of $a_{k,n}$ elements of $Y_n$ satisfying a set $S_{k,n}$ 
of sentences, including those of $R_{k,n}$ and those needed to express the fact that $\partial\Delta[n]$
is sent to the basepoint~$f(p)$. Homotopies into ${\rm Ex}^k Y$ are formalized similarly.

The assertion that $f\colon X\to Y$ induces $\pi_n(\Exinf X,p)\cong\pi_n(\Exinf Y,f(p))$ 
for every $p\in X_0$ can therefore be expressed by stating that for every $k<\omega$ and every map
$\beta\colon\Sph^n\to{\rm Ex}^kY$ based at $f(p)$
there exist $l<\omega$ and a map $\alpha\colon\Sph^n\to{\rm Ex}^lX$
based at $p$ and a homotopy $H\colon \Sph^n\wedge\Delta[1]_+\to {\rm Ex}^r Y$
from $({\rm Ex}^r f)\circ\alpha$ to~$\beta$, where $r\ge k$ and $r\ge l$, 
and, moreover, if $\alpha'\colon\Sph^n\to{\rm Ex}^m X$ 
is based at $p$ and there is a homotopy from $({\rm Ex}^r f)\circ\alpha'$ to $\beta$ with $r\ge k$
and $r\ge m$, then there is a homotopy $H\colon \Sph^n\wedge\Delta[1]_+\to {\rm Ex}^s X$ 
from $\alpha$ to $\alpha'$ with $s\ge l$ and $s\ge m$.
Therefore, the class of
weak equivalences between simplicial sets 
is $\Delta_1$\nobreakdash-definable.

Having proved (1) and~(2), we next address~(3).
A~spectrum $F$ is weakly contractible if and only if
all its homotopy groups vanish, that is,
\[
\mbox{$\colim_n\, [\Sph^{n+k}, F_n]=0$ for all~$k\in\Z$.}
\]
This is equivalent to imposing that, for all $k\in\Z$ and $n\ge 0$ such that $n+k\ge 0$,
each pointed map 
$\beta\colon \Sph^{n+k}\to\Exinf F_n$ 
becomes nullhomotopic after suspending it a finite number of times (say, $m$~times) 
and composing with the structure maps $\sigma_n\colon S F_n\to F_{n+1}$.
More precisely, on the one hand, we have:
\begin{equation}
\label{suspending}
\SelectTips{cm}{}
\xymatrix{ 
\Sph^{n+m+k}\ar[r]^-{S^m \beta} & S^m \Exinf F_n\ar[r]^-{j} & \Exinf S^m \Exinf F_n,
}
\end{equation}
and, on the other hand, there are maps
\[
\SelectTips{cm}{}
\xymatrixcolsep{2cm}
\xymatrix{ 
\Exinf S^m \Exinf F_n & \Exinf S^m F_n\ar[l]_-{\Exinf S^m j}\ar[r]^-{\Exinf \sigma} & \Exinf F_{n+m},
}
\]
where $\sigma$ is an abbreviation for
$\sigma_{n+m-1}\circ S\sigma_{n+m-2}\circ \cdots\circ S^{m-2}\sigma_{n+1}\circ S^{m-1}\sigma_n$.
The maps $j$ and $\Exinf S^m j$ are natural weak equivalences.

Hence, $F$ is weakly contractible if and only if, for each~$k\in\Z$
and each $(n+k)$\nobreakdash-simplex $x\in\Exinf F_n$ whose faces are equal to the basepoint,
there is an $(n+m+k)$\nobreakdash-simplex $y\in \Exinf S^m F_n$ whose faces are equal to the basepoint and
an $(n+m+k+1)$\nobreakdash-simplex $z\in \Exinf F_{n+m}$ whose top face 
is $y$ and all its other faces are equal to the basepoint, and $(\Exinf S^m j)y\sim j(S^mx)$.

We finally prove~(4).
In order to formalize the fact that a spectrum $E$ is an $\Omega$\nobreakdash-spectrum, 
we first need that each simplicial set $E_n$ be fibrant. Then we need to define
the adjoint maps $\tau_n\colon E_n\to\Omega E_{n+1}$ and we need to impose that each $\tau_n$
be a weak equivalence. To define~$\tau_n$, let $x$ be a $k$\nobreakdash-simplex of~$E_n$. Its image
in $\Omega E_{n+1}=\Map_*(\Sph^1,E_{n+1})$ is a map $\Sph^1\wedge\Delta[k]_+\to E_{n+1}$
which is determined by imposing that 
\[
(\tau_n(x))(se_1,e_k)=\sigma_n(se_1,x),
\] 
where $e_1$ is the nondegenerate $1$\nobreakdash-simplex of $\Sph^1$ and $e_k$ is the nondegenerate
$k$\nobreakdash-simplex of~$\Delta[k]$, and $s$ denotes a composition of degeneracies.
\end{proof}

In what follows, we denote by $\Ssets_*$ the category of pointed simplicial sets
and pointed maps.

\begin{theorem}
\label{homology}
The class of $E_*$\nobreakdash-acyclic simplicial sets for a spectrum $E$ is~{\boldmath$\Delta_1$} with
$E$ as a parameter.
\end{theorem}

\begin{proof}
If $(X,p)$ and $(Y,q)$ are pointed simplicial sets, then
$W=X\vee Y$ is a pointed simplicial set contained in $X\times Y$ such that $W_n$ contains
all elements of the form $(x,sq)$ with $x\in X_n$ and all those of the form $(sp,y)$ with $y\in Y_n$,
where $s$ is a composition of degeneracies, with basepoint $(p,q)$.
The smash product $X\wedge Y$ is obtained from $X\times Y$ by collapsing $X\vee Y$ to a point.
Hence, $(X\wedge Y)_n=(X_n\times Y_n)\setminus (W_n\setminus \{(sp,sq)\})$ for all~$n$,
and we declare equal to $(sp,sq)$ all faces of elements of $X_{n+1}\times Y_{n+1}$ 
and all degeneracies of elements of $X_{n-1}\times Y_{n-1}$ taking values in~$W_n$.

If $(X,p)$ is a pointed simplicial set and $E$ is a spectrum with structure maps $\langle \sigma_n : 0\le n<\omega\rangle$,
then $X\wedge E$ is a spectrum with $(X\wedge E)_n=X\wedge E_n$ and structure maps $({\rm id}\wedge \sigma_n)\circ(\tau\wedge{\rm id})$
for all~$n$, where $\tau\colon \Sph^1\wedge X\to X\wedge\Sph^1$ is the twist map. 
By part~(3) of Proposition~\ref{weqs}, the statement that $X\wedge E$ is weakly contractible
is~$\Delta_1$.
However, a formula expressing this fact has to contain
a definition of $X\wedge E$, where $E$ is a given spectrum treated as a parameter. 
This can be done in two equivalent ways, as follows:
\begin{equation}
\label{first}
\begin{array}{c}
X\in\Ssets_*\,\wedge\,\exists F\, [\mbox{$F$ is a spectrum}\,\wedge\, (\forall n<\omega)((F_n=X\wedge E_n)
\\[0.1cm]
\wedge \,
\sigma_n^F=({\rm id}\wedge \sigma_n^E)\circ(\tau\wedge{\rm id}))\,\wedge\, \mbox{$F$ is weakly contractible}];
\end{array}
\end{equation}
\begin{equation}
\label{second}
\begin{array}{c}
X\in\Ssets_*\,\wedge\,\forall F\, [[\mbox{$F$ is a spectrum}\,\wedge\, (\forall n<\omega)((F_n=X\wedge E_n)
\\[0.1cm]
\wedge\,
\sigma_n^F=({\rm id}\wedge \sigma_n^E)\circ(\tau\wedge{\rm id}))]\,\to\, \mbox{$F$ is weakly contractible}].
\end{array}
\end{equation}
Since \eqref{first} is $\Sigma_1$ and \eqref{second} is~$\Pi_1$, the theorem is proved.
\end{proof}

As explained in Section~\ref{Levyhierarchy}, the fact that homological acyclic classes
are {\boldmath$\Delta_1$} implies that they are absolute. This means that,
if $E$ is a spectrum and $M$ is a transitive model of ZFC such that $E\in M$ 
(in which case $E$ is a spectrum in $M$ as well, since being a spectrum is~$\Delta_0$), then
a simplicial set $X\in M$ is $E_*$\nobreakdash-acyclic in $M$ if and only if it is $E_*$\nobreakdash-acyclic.

We thank Federico Cantero for pertinent remarks about the argument given in the proof
of the next result.

\begin{theorem}
\label{cohomology}
The class of $E^*$\nobreakdash-acyclic simplicial sets 
for an $\Omega$\nobreakdash-spectrum~$E$ is~{\boldmath$\Delta_2$}
with $E$ as a parameter.
\end{theorem}

\begin{proof}
Let $E$ be an $\Omega$\nobreakdash-spectrum, which will be used as a parameter.
By part~(4) of Proposition~\ref{weqs}, every transitive model of ZFC containing $E$
will agree with the fact that $E$ is an $\Omega$\nobreakdash-spectrum.

A~simplicial set $X$ is $E^*$\nobreakdash-acyclic if and only if, for all $k\in\Z$ and $n\ge0$ with $n+k\ge 0$, every map
$S^n X\to E_{n+k}$ becomes nullhomotopic after suspending it a finite
number of times and composing with the structure maps of $E$ as in~\eqref{suspending}. 
This claim leads to a $\Pi_2$ formula ---note that a map $S^n X\to E_{n+k}$
is no longer determined by any finite set of simplices of~$E_{n+k}$. 
Next we show that it is possible to restate it by means of a $\Sigma_2$ formula.

A pointed simplicial set $(X,p)$ is $E^*$\nobreakdash-acyclic if and only if for all~$n<\omega$ the
simplicial set $\Map_*(X,E_n)$ is weakly contractible, assuming that $E$ is an $\Omega$\nobreakdash-spectrum.
Thus, $X$ is $E^*$\nobreakdash-acyclic if and only if the following formula holds, where we need to define
$M=\Map_*(X,E_n)$:
\[
\begin{array}{cc}
X\in\Ssets_* \, \wedge \, (\forall n<\omega)\, \exists M\, [M\in\Ssets_*
\\[0.1cm]
\wedge \,
(\forall k<\omega)\, [(\forall f\in M_k)\, f\in\Ssets_*(X\wedge\Delta[k]_+,E_n)
\\[0.1cm]
\wedge \,
\forall g\,(g\in\Ssets_*(X\wedge\Delta[k]_+,E_n)\to g\in M_k)]\,\wedge\,
\mbox{$M$ is weakly contractible}].
\end{array}
\]
According to Proposition~\ref{weqs}, this is a $\Sigma_2$ formula.
\end{proof}

In order to state and prove the next results, we use the term \emph{homotopy reflection}
(also called \emph{homotopy localization} elsewhere)
to designate a functor $L\colon\Ssets_*\to\Ssets_*$ equipped with a natural transformation
$\eta\colon{\rm Id}\to L$ which preserves
weak equivalences and becomes a reflection when passing to the homotopy category.
For a homotopy reflection~$L$, an \emph{$L$\nobreakdash-equiv\-al\-ence}
is a map $f\colon X\to Y$ such that $Lf\colon LX\to LY$ is an isomorphism in the homotopy category,
and a simplicial set $X$ is called \emph{$L$\nobreakdash-local}
if it is fibrant and weakly equivalent to $LX$ for some~$X$.

We also recall that, for a pointed map~$f\colon A\to B$,
a connected fibrant simplicial set $X$ is \emph{$f$\nobreakdash-local} if the induced map
of pointed function complexes
\[
\Map_*(f,X):\Map_*(B,X)\longrightarrow\Map_*(A,X)
\]
is a weak equivalence,
and a nonconnected $X$ is $f$\nobreakdash-local if each of its
connected components is $f$\nobreakdash-local with any choice of basepoint;
cf.~\cite[1.A.1]{Far2}.
Note that, if $X$ is $f$\nobreakdash-local for a map $f\colon A\to B$, then $f$
induces a bijection $[B,X]\cong [A,X]$, since $[B,X]$ is in natural bijective correspondence with 
the set of connected components of $\Map_*(B,X)$. Hence, being $f$\nobreakdash-local is a stronger
condition than being orthogonal to $f$ in the homotopy category.

The same terminology is used for a set or a proper
class of maps~$\Fe$; that is, a simplicial set is \emph{$\Fe$\nobreakdash-local} if it is
$f$\nobreakdash-local for all $f\in\Fe$.
An \emph{$\Fe$\nobreakdash-localization} is a homotopy
reflection $L$ such that the class of $L$\nobreakdash-local spaces coincides with
the class of $\Fe$\nobreakdash-local spaces.

\begin{lemma}
\label{flocal-slocal}
Given any class of pointed maps~$\Se$ between simplicial sets, if there is a subclass $\Fe\subseteq\Se$ such that
each element of $\Se$ is a filtered colimit of elements of~$\Fe$, then
every $\Fe$\nobreakdash-local space is $\Se$\nobreakdash-local. 
\end{lemma}

\begin{proof}
The argument is analogous to the one used in the proof of Lemma~\ref{dense}.
Let $f\colon A\to B$ be any element of $\Se$ and let $X$ be an $\Fe$\nobreakdash-local
simplicial set, which we may assume connected. Write $f=\colim\, f_k$ (in the category of pointed maps between
simplicial sets), where $f_k\colon A_k\to B_k$ is in $\Fe$ for all $k\in\Ke$, and $\Ke$ is filtered.
Now we use, as in~\cite[Lemma~5.2]{CSS}, the fact that the natural map
\[ {\rm hocolim}\, f_k\longrightarrow \colim\, f_k \] 
is a weak equivalence, since homotopy groups commute with filtered colimits
(here ${\rm hocolim}$ is a pointed homotopy colimit \cite[18.8]{Hir}).
Hence,
\begin{align*}
\Map_*(B,X) & \simeq \Map_*({\rm hocolim}\,B_k,X) \simeq {\rm holim}\,\Map_*(B_k,X) \\
& \simeq {\rm holim}\,\Map_*(A_k,X) \simeq \Map_*({\rm hocolim}\,A_k,X) \simeq \Map_*(A,X),
\end{align*}
from which it follows indeed that $X$ is $\Se$\nobreakdash-local.
\end{proof}

\begin{theorem}
\label{localizationsexist}
Assume the existence of arbitrarily large supercompact
cardinals. Then for every additive cohomology theory $E^*$
defined on simplicial sets there is a homotopy reflection $L$ such that the $L$\nobreakdash-equivalences 
are precisely the $E^*$\nobreakdash-equivalences.
\end{theorem}

\begin{proof}
Let $\Se$ be the class of $E^*$\nobreakdash-equivalences for a given
additive cohomology theory~$E^*$, and view it as a full
subcategory of the category of pointed maps between simplicial sets,
which is accessibly embedded into a category of structures, by Lemma~\ref{relativizing}.
Since the class of $E^*$\nobreakdash-equivalences coincides with the class of
maps whose mapping cone is $E^*$\nobreakdash-acyclic, Theorem~\ref{cohomology}
tells us that $\Se$ is~{\boldmath$\Delta_2$}, hence~{\boldmath$\Sigma_2$}.
Consequently, it follows from Theorem~\ref{maintheorem} that
there is a regular cardinal $\kappa$ and a set $\Fe$ of
$E^*$\nobreakdash-equivalences
such that every $E^*$\nobreakdash-equivalence is a $\kappa$\nobreakdash-filtered colimit 
of elements of~$\Fe$ in the category of pointed maps between simplicial sets.

To conclude the proof, let $f\colon A\to B$ be the coproduct of all the elements of~$\mathcal{F}$,
and let $L$ be $f$\nobreakdash-localization, as constructed in~\cite{BouJPAA}, \cite{Far2}
or~\cite{Hir}. Since all the elements of
$\mathcal{F}$ are $E^*$\nobreakdash-equivalences and $E^*$ is additive, $f$ is an $E^*$\nobreakdash-equivalence. 

Let $E$ be an $\Omega$\nobreakdash-spectrum representing~$E^*$. Since $f$ is an $E^*$\nobreakdash-equivalence, 
it induces bijections
$[B,E_n]\cong[A,E_n]$ for all~$n$, and in fact weak equivalences $\Map_*(B,E_n)\simeq\Map_*(A,E_n)$
for all~$n$. In other words, the basepoint component of $E_n$ is $f$\nobreakdash-local for all~$n$.
Since $E_n$ is a loop space, all its connected components have the same homotopy type and
therefore $E_n$ itself is $f$\nobreakdash-local for all~$n$. It follows that every $L$\nobreakdash-equivalence
$g\colon X\to Y$ induces a weak equivalence $\Map_*(Y,E_n)\simeq\Map_*(X,E_n)$ for all~$n$, and we conclude that all $L$\nobreakdash-equivalences are $E^*$\nobreakdash-equivalences.

Conversely, every $E^*$\nobreakdash-equivalence is, as said above, a $\kappa$\nobreakdash-filtered
colimit of objects from~$\Fe$. According to Lemma~\ref{flocal-slocal}, 
every $L$\nobreakdash-local simplicial set is $E^*$\nobreakdash-local, and therefore
all $E^*$\nobreakdash-equivalences are $L$\nobreakdash-equivalences. This completes the argument.
\end{proof}

What we have proved is that localization with respect to any additive cohomology
theory exists on the homotopy category of simplicial sets if arbitrarily
large supercompact cardinals exist. This is a substantial improvement of 
\cite[Corollary~5.4]{CSS}, where it was proved that the existence of 
cohomological localizations follows from~Vop\v{e}nka's principle. 

We also emphasize that from Theorem~\ref{homology} it follows, by a
similar method as in the proof of Theorem~\ref{localizationsexist} (or using
Theorem~\ref{invertingclassesofmaps} below), that the existence
of \emph{homological} localizations (for representable homology theories) 
is provable in~ZFC. Bousfield did it indeed in~\cite{BouTop}.

The same line of argument provides an answer to Farjoun's question in \cite{Far1} of whether
all homotopy reflections are $f$\nobreakdash-local\-izations for some map~$f$. It was shown in \cite{CSS}
that the answer is affirmative under Vop\v{e}nka's principle, and Prze\'{z}dziecki proved in \cite{Prz}
that an affirmative answer is in fact equivalent to Vop\v{e}nka's principle.
Here we prove an analogue of Theorem~\ref{corollary3.4}.

\begin{theorem}
\label{farjoun}
A homotopy reflection $L$ on simplicial sets is an $f$\nobreakdash-localiza\-tion for some map $f$ under 
any of the following assumptions:
\begin{itemize}
\item[(1)]
The class of $L$\nobreakdash-equivalences is definable with sufficiently low complexity.
\item[(2)]
The class of $L$\nobreakdash-local simplicial sets is~{\boldmath$\Sigma_{n+1}$}
for $n\ge 2$ and there is a proper class of $\Cn$\nobreakdash-extendible cardinals.
\end{itemize}
\end{theorem}

\begin{proof}
For (1), we may choose, by~Theorem~\ref{maintheorem}, a set $\Fe$ of $L$\nobreakdash-equivalences such that 
every $L$\nobreakdash-equivalence is a filtered colimit of elements of~$\Fe$ in the category of pointed
maps between simplicial sets.
Let $f$ be the coproduct of all the elements of~$\Fe$.
Then $f$ is an $L$\nobreakdash-equivalence, since the class of $L$\nobreakdash-equivalences
is closed under coproducts. Therefore, every $L$\nobreakdash-local simplicial set is $f$\nobreakdash-local,
by \cite[Corollary~4.4]{CSS}. Conversely, every $f$\nobreakdash-local simplicial set is $L$\nobreakdash-local
by Lemma~\ref{flocal-slocal}.

In order to prove~(2), note that, if the class of $L$\nobreakdash-local simplicial sets is 
{\boldmath$\Sigma_{n+1}$},
then the class of $L$\nobreakdash-equivalences
is {\boldmath$\Pi_{n+1}$},
since $f\colon A\to B$ is an $L$\nobreakdash-equivalence if and only if the induced function 
$[B,X]\rightarrow [A,X]$
is a bijection for each $L$\nobreakdash-local space~$X$, which can be formalized as
\[
\begin{array}{c}
\forall X\,\forall g\, [(\text{$X$ is an $L$\nobreakdash-local simplicial set}\,\wedge\,
g\in\Ssets_*(A,X))\to \\[0.1cm]
(\exists h\,(h\in\Ssets_*(B,X) \wedge h\circ f\simeq g)\,\wedge \,
\text{any two such maps are homotopic})].
\end{array}
\]
The statement ``any two such maps are homotopic'' can be formally written as a $\Pi_2$ formula.
Hence the same argument as in part~(1) applies under the
assumption that a proper class of $\Cn$\nobreakdash-extendible cardinals exists,
by means of Theorem~\ref{maintheorem}.
\end{proof}

The corresponding analogue of Theorem~\ref{corollary6.5} is the next result.
Localization with respect to proper classes of maps was shown to exist in \cite{Ch1}
under restrictive conditions.

\begin{theorem}
\label{invertingclassesofmaps}
Let $\Se$ be any (possibly proper) class of maps of simplicial sets.
If $\Se$ is definable with sufficiently low complexity, then an $\Se$\nobreakdash-localization exists.
\end{theorem}

\begin{proof}
Theorem~\ref{maintheorem} implies that there is a set $\Fe\subseteq\Se$ such that
every $f\in\Se$ is a filtered colimit of elements of~$\Fe$.
Then $\Fe$\nobreakdash-localization exists since $\Fe$ is a set, and
every $\Fe$\nobreakdash-local simplicial set is $\Se$\nobreakdash-local
by Lemma~\ref{flocal-slocal}. Since $\Fe\subseteq\Se$, all $\Se$\nobreakdash-local simplicial sets are $\Fe$\nobreakdash-local, so the proof is complete.
\end{proof}

\section{Bergman's question}
\label{Bergmanquestion}

If $\Sigma$ is a finitary operational signature, then
$\Sigma$\nobreakdash-structures are \emph{universal algebras}.
If $\Ce$ is a full subcategory of ${\bf Str}\,\Sigma$ and $n$ is a nonnegative
integer, an \emph{$n$\nobreakdash-ary implicit operation} $f$ on $\Ce$ is a 
natural transformation from the $n$\nobreakdash-fold product functor to the identity functor; that is,
a collection of maps $f_X\colon X^n\to X$
indexed by objects $X$ of $\Ce$ such that the square
\[
\xymatrix@=3pc{
X^n \ar[r]^{h^n} \ar[d]_{f_X} & Y^n \ar[d]^{f_Y} \\
X \ar[r]^{h} & Y
}
\]
commutes for each homomorphism $h\colon X\to Y$. Such implicit operations are very useful in finite universal 
algebra; see~\cite{A}. If $\Ce$ is a proper class with no homomorphisms except identities, then 
each collection $\{f_X\}_{X\in{\Ce}}$ is an implicit operation. Thus, assuming the negation of 
Vop\v{e}nka's principle, there is a proper class of implicit operations on~$\Ce$. In connection 
with~\cite{B}, Bergman asked whether this can happen assuming Vop\v{e}nka's principle.

\begin{theorem}
\label{th8.1}
For a finitary operational signature $\Sigma$,
Vop\v{e}nka's principle implies that there is only a set of implicit operations on each full subcategory of ${\bf Str}\,\Sigma$.
\end{theorem}
\begin{proof}
Let $\Ce$ be a full subcategory of ${\bf Str}\,\Sigma$, where $\Sigma$ is $S$\nobreakdash-sorted. By~\cite{AR2}, Vop\v{e}nka's principle~implies 
that there is a regular
cardinal $\kappa$ and a set $\Ae$ of objects in $\Ce$ such that each object of $\Ce$ 
is a $\kappa$\nobreakdash-filtered colimit of objects of~$\Ae$. Since the forgetful functor 
${\bf Str}\,\Sigma \to\Sets^S$
and the $n$\nobreakdash-fold product functor $(-)^n\colon \Sets^S\to\Sets^S$ preserve colimits, each implicit operation
$f_X$ with $X\in\Ce$ is uniquely determined 
by $\{f_A\}_{A\in{\Ae}}$. Hence there is only a set 
of distinct implicit operations on~$\Ce$.
\end{proof} 

We improve this result as follows.

\begin{theorem}
\label{th8.2}
For a finitary operational signature~$\Sigma$,
every full subcategory $\Se$ of ${\bf Str}\,\Sigma$ definable with sufficiently low complexity
has only a set of implicit operations.
\end{theorem}
\begin{proof}
As shown in the proof of Theorem~\ref{maintheorem}, for each object $Y$ of~$\Se$
the slice category $(\Se\cap H(\kappa)\downarrow Y)$ 
is cofinal in $(\Ke\downarrow Y)$ for some regular 
cardinal~$\kappa$, where $\Ke$ is the (essentially small) class
of $\kappa$\nobreakdash-presentable objects in ${\bf Str}\,\Sigma$. Thus each object
of $\Se$ is a $\kappa$\nobreakdash-filtered colimit of objects from the set
$\Se\cap H(\kappa)$. The rest is the same as in the proof of Theorem~\ref{th8.1}.
\end{proof}

\bigskip

{\footnotesize

\noindent
Joan Bagaria,
ICREA (Instituci\'o Catalana de Recerca i Estudis Avan\c{c}ats) and
Departament de L\`ogica, Hist\`oria i Filosofia de la Ci\`encia, Universitat de Barcelona,
Montalegre~6, 08001~Barcelona, Spain,
{\tt joan.bagaria@icrea.cat; bagaria@ub.edu}.

\bigskip

\noindent
Carles Casacuberta,
Departament d'\`Algebra i Geometria and Institut de Matem\`atica, 
Universitat de Barcelona,
Gran Via de les Corts Catalanes~585, 08007~Barcelona, Spain,
{\tt carles.casacuberta@ub.edu}.

\bigskip

\noindent
A. R. D. Mathias,
ERMIT, Universit\'e de la R\'eunion,
UFR Sciences et Technologies,
Labo\-ratoire d'Informatique et de Math\'ematiques,
2~rue~Joseph Wetzel, B\^atiment~2,
F\nobreakdash-97490 Sainte Clotilde, France outre-mer,
{\tt ardm@univ-reunion.fr; ardm@dpmms.cam.ac.uk}.

\bigskip

\noindent
Ji\v{r}\'{\i} Rosick\'{y},
Department of Mathematics and Statistics, Masaryk University,
Kotl\'{a}\v{r}sk\'{a}~2, 600~00~Brno, Czech Republic,
{\tt rosicky@math.muni.cz}.

}


\begin{thebibliography}{99}

\bibitem{AHS} J. Ad\'amek, H. Herrlich, and G. Strecker, \textit{Abstract and Concrete Categories},
John Wiley, New York, 1990. Reprinted in \textit{Repr. Theory Appl. Categ.} {\bf 17} (2006).

\bibitem{AR} J. Ad\'amek and J. Rosick\'y, \textit{Locally Presentable and Accessible Categories}, 
London Math. Soc. Lecture Note Ser., vol.~189, Cambridge University Press, Cambridge, 1994.

\bibitem{AR2} J. Ad\' amek and J. Rosick\' y, On preaccessible categories, \textit{J. Pure
Appl. Algebra} \textbf{105} (1995), 225--232.

\bibitem{Ada} J. F. Adams, A variant of E.~H.~Brown's representability theorem,
\textit{Topology} \textbf{10} (1971), 185--198.

\bibitem{AF} J. F. Adams, Localisation and completion, with an addendum on the
use of Brown--Peterson homology in stable homotopy, Lecture notes by Z.~Fiedorowicz
on a course given at The University of Chicago in Spring 1973. Revised and
supplemented by Z.~Fiedorowicz, 2010, arXiv:1012.5020.

\bibitem{A} J. Almeida, \textit{Finite Semigroups and Universal Algebra}, World Scientific, 1994.

\bibitem{Ba} J. Bagaria, $C^{(n)}$ cardinals, \textit{Arch. Math. Logic} \textbf{51} (2012), 213--240.

\bibitem{BBT} J. Bagaria and A. Brooke-Taylor, On colimits and elementary embeddings, preprint,
arXiv:1202.5215.

\bibitem{BCM} J. Bagaria, C. Casacuberta, and A. R. D. Mathias, 
Epireflections and supercompact cardinals, \textit{J. Pure Appl. Algebra} 
\textbf{213} (2009), 1208--1215.

\bibitem{B} G. M. Bergman, \textit{An Invitation to General Algebra and Universal Constructions}, 
Henry Helson, 1998.

\bibitem{BouTop} A. K. Bousfield, The localization of spaces with
respect to homology, \textit{Topology} \textbf{14} (1975), 133--150.

\bibitem{BouJPAA} A. K. Bousfield, Constructions of factorization systems in categories,
\textit{J. Pure Appl. Algebra} {\bf 9} (1976/77), 207--220.

\bibitem{BF} A. K. Bousfield and E. M. Friedlander,
Homotopy theory of $\Gamma$-spaces, spectra, and bisimplicial
sets, in: \textit{Geometric Applications of Homotopy Theory}, 
Lecture Notes in Math., vol.~658,
Springer, Berlin, Heidelberg, 1978, 80--130.

\bibitem{BT} A. Brooke-Taylor, Indestructibility of Vop\v{e}nka's Principle, \textit{Arch. Math. Logic}
\textbf{50} (2011), 515--529.

\bibitem{CCh} C. Casacuberta and B. Chorny, The orthogonal subcategory problem in homotopy theory,
in: \textit{An Alpine Anthology of Homotopy Theory}, 
Contemp. Math., vol.~399, Amer. Math. Soc., Providence, 2006, 41--53.

\bibitem{CGR} C. Casacuberta, J. J. Guti\'errez, and J. Rosick\'y, Are all localizing
subcategories of stable homotopy categories coreflective?, preprint, arXiv:1106-2218.

\bibitem{CSS} C. Casacuberta, D. Scevenels, and J. H. Smith, Implications of large-cardinal principles 
in homotopical localization, \textit{Adv. Math.} \textbf{197} (2005), 120--139.

\bibitem{Ch1} B. Chorny, Localization with respect to a class of maps~I --
Equivariant localization of diagrams of spaces, \textit{Israel J. Math.} \textbf{147} (2005), 93--139.

\bibitem{Ch2} B. Chorny, Abstract cellularization as a cellularization with respect to a set of objects, 
in: \textit{Categories in Algebra, Geometry and Mathematical Physics}, 
Contemp. Math., vol.~431, Amer. Math. Soc., Providence, 2007, 165--170.

\bibitem{Far1} E. Dror Farjoun, Homotopy localization and $v_1$-periodic spaces,
in: \textit{Algebraic Topology; Homotopy and Group Cohomology},
Lecture Notes in Math., vol.~1509, Springer, Berlin, Heidelberg, 1992, 104--113.

\bibitem{Far2} E. Dror Farjoun, Cellular Spaces, Null Spaces and
Homotopy Localization, Lecture Notes in Math., vol.~1622,
Springer, Berlin, Heidelberg, 1996.

\bibitem{E} K. Eda, A Boolean power and a direct product of abelian groups, \textit{Tsukuba J.
Math.} \textbf{6} (1982), 187--193.

\bibitem{EM} P. C. Eklof and A. H. Mekler, \textit{Almost Free Modules: Set-theoretic Methods}, North-Hol\-land, Amsterdam, 1990.
Revised edition, North-Holland Mathematical Library, vol.~65, Elsevier, Amsterdam, 2002.

\bibitem{F} P. J. Freyd, Homotopy is not concrete, in: \textit{The Steenrod Algebra and its Applications},
Lecture Notes in Math., vol.~168, Springer, Berlin, Heidelberg, 1970, 25--34. 
Reprinted in \textit{Repr. Theory Appl. Categ.} \textbf{6} (2004), 1--10.

\bibitem{FK} P. J. Freyd and G. M. Kelly, Categories of continuous functors I,
\textit{J. Pure Appl. Algebra} \textbf{2} (1972), 169--191.

\bibitem{GU} P. Gabriel and F. Ulmer, \textit{Local pr\"asentierbare Kategorien},
Lecture Notes in Math., vol.~221, Springer, Berlin, Heidelberg, 1971.

\bibitem{GJ} P. G. Goerss and J. F. Jardine, Simplicial Homotopy Theory,
Progress in Math., vol.~174, Birkh\"auser, Basel, 1999; second printing:
Modern Birkh\"auser Classics, 2009.

\bibitem{Hir} P. S. Hirschhorn, \textit{Model Categories and Their Localizations},
Math. Surveys Monographs, vol.~99, Amer. Math. Soc., Providence, 2003.

\bibitem{Hov} M. Hovey, Cohomological Bousfield classes,
\textit{J. Pure Appl. Algebra} \textbf{103} (1995), 45--59. 

\bibitem{J1} T. Jech, \textit{Set Theory},
Pure and Applied Math., Academic Press, New York, 1978.

\bibitem{J2} T. Jech, \textit{Set Theory. The Third Millenium Edition, Revised and Expanded},
Springer Monographs in Math., Springer, Berlin, Heidelberg, 2003.

\bibitem{Kan} D. M. Kan, On c.s.s. complexes, \textit{Amer. J. Math.} 
\textbf{79} (1957), 449--476.

\bibitem{K} A. Kanamori, \textit{The Higher Infinite: Large Cardinals in Set
Theory from Their Beginnings}, Perspectives in Mathematical Logic, Springer,
Berlin, Heidelberg, 1994.

\bibitem{Kun1} K. Kunen, Elementary embeddings and infinitary combinatorics, 
\textit{J. Symbolic Logic} \textbf{36} (1971), 407--413.

\bibitem{Kun2} K. Kunen, \textit{Set Theory: An Introduction to Independence Proofs},
Studies in Logic and the Foundations of Math., vol.~102, Elsevier, Amsterdam, 1980.

\bibitem{Mac} S. Mac Lane, \textit{Categories for the Working Mathematician},
Graduate Texts in Math., vol.~5, Springer, New York, 1998 (2nd ed.).

\bibitem{M} M. Magidor, On the role of supercompact and extendible cardinals in logic, 
\textit{Israel J. Math.} \textbf{10} (1971), 147--157.

\bibitem{MP} M. Makkai and R. Par\'e, \textit{Accessible Categories: The Foundations of Categorical 
Model Theory}, Contemp. Math., vol.~104, Amer. Math. Soc., Providence, 1989.

\bibitem{ARDM} A. R. D. Mathias, Weak systems of Gandy, Jensen and Devlin, in: 
\textit{Set Theory, Centre de Recerca Matem\`atica, Barcelona, 2003--2004}, 
Trends in Mathematics, Birkh\"auser, Basel, 2006, 149--224 

\bibitem{May} J. P. May, \textit{Simplicial Objects in Algebraic Topology}, 
The University of Chicago Press, Chicago, 1967.

\bibitem{Ne} J. Ne\v{s}et\v{r}il,  A rigid graph for every set, 
\textit{J. of Graph Theory} \textbf{39} (2002), 108--110.

\bibitem{Prz} A. J. Prze\'{z}dziecki, An ``almost'' full embedding of the category of
graphs into the category of groups, \textit{Adv. Math.} \textbf{225} (2010), 1893--1913.

\bibitem{RT} J. Rosick\'y and W. Tholen, Left-determined model categories and universal
homotopy theories, \textit{Trans. Amer. Math. Soc.} \textbf{355} (2003), 3611--3623.

\bibitem{Rud} Yu. B. Rudyak, \textit{On Thom Spectra, Orientability, and Cobordism},
Springer Monographs in Mathematics, Springer, Berlin, 1998.

\bibitem{TAR} V. Trnkov\'{a}, J. Ad\'amek, and J. Rosick\'y, Topological reflections revisited, 
\textit{Proc. Amer. Math. Soc.} \textbf{108} (1990), 605--612.

\end{thebibliography}
\end{document}